\documentclass{tac}
\usepackage{amsmath,amssymb,amsfonts}
\usepackage[mathscr]{eucal}
\usepackage[margin=1.15in]{geometry}

  \usepackage{multirow}
\usepackage{arydshln}
\usepackage{adjustbox}

\usepackage{tikz}
\usetikzlibrary{matrix,arrows,decorations.pathmorphing,cd,patterns,calc,backgrounds,patterns}
\tikzset{dummy/.style= {circle,fill,draw,inner sep=0pt,minimum size=1.2mm}}
\tikzset{vertex/.style={fill, circle, minimum size=.1cm, inner sep=0pt}}

\usepackage{multirow}

\usepackage[textwidth=25mm, textsize=tiny]{todonotes}
\usepackage[final]{microtype}

\usepackage{enumerate}

\numberwithin{equation}{section} 
\numberwithin{figure}{section}

\newcommand{\newrefformat}[2]{}
\usepackage[nameinlink,capitalise,noabbrev]{cleveref}

\crefname{lemma}{Lemma}{Lemmas}
\crefname{theorem}{Theorem}{Theorems}
\crefname{definition}{Definition}{Definitions}
\crefname{proposition}{Proposition}{Propositions}
\crefname{remark}{Remark}{Remarks}
\crefname{corollary}{Corollary}{Corollaries}
\crefname{equation}{Equation}{Equations}
\crefname{construction}{Construction}{Constructions}
\crefname{ex}{Example}{Examples}
\crefname{appsec}{Appendix}{Appendices}
\crefname{section}{Section}{Sections}
\crefname{subsection}{Subsection}{Subsections}

\newcommand{\TT}{\mathbb{T}}

\newcommand{\abs}[1]{\lvert#1\rvert}

\newcommand{\Aut}{\operatorname{Aut}}

\newcommand{\Cat}{\mathcal Cat}
\newcommand{\cat}[1]{\mathcal{#1}}

\newcommand{\Fin}{\mathcal Fin}
\newcommand{\Fun}{\operatorname{Fun}}

\newcommand{\Hom}{\operatorname{Hom}}
\newcommand{\id}{\operatorname{id}}
\newcommand{\im}{\operatorname{im}}

\newcommand{\Lie}{\operatorname{Lie}}
\newcommand{\longto}{\longrightarrow}

\newcommand{\ob}{\operatorname{ob}}
\newcommand{\onto}{\twoheadrightarrow}

\newcommand{\ord}{\operatorname{ord}}
\renewcommand{\phi}{\varphi}
\newcommand{\Set}{\mathcal Set}
\newcommand{\set}[1]{\{ #1 \}}

\newcommand{\sSet}{s\mathcal Set}

\newcommand{\Top}{\mathcal Top}

\input xy
\xyoption{all}

\title{Equivariant trees and partition complexes}

\author[J.E.\ Bergner, P.\ Bonventre, M.E.\ Calle, D.\ Chan, and M.\ Sarazola]{Julia E.\ Bergner, Peter Bonventre, Maxine E.\ Calle, David Chan, and Maru Sarazola}

\address{Department of Mathematics, University of Virginia,\\ Charlottesville, VA 22904\\[5pt]
Department of Mathematics and Statistics, Georgetown University,\\ Washington, DC 20007\\[5pt]
Department of Mathematics, University of Pennsylvania,\\ Philadelphia, PA 19104\\[5pt]
Department of Mathematics, Michigan State University,\\ East Lansing, MI 48824\\[5pt]
School of Mathematics, University of Minnesota,\\ Minneapolis MN, 55455}


\begin{document}
\maketitle
\begin{abstract}
    We introduce two definitions of $G$-equivariant partitions of a finite $G$-set, both of which yield $G$-equivariant partition complexes. By considering suitable notions of equivariant trees, we show that $G$-equivariant partitions and $G$-trees are $G$-homotopy equivalent, generalizing existing results for the non-equivariant setting.  Along the way, we develop equivariant versions of Quillen's Theorems A and B, which are of independent interest.
\end{abstract}

\section{Introduction}

Given a finite set $\mathbf n=\{1, \ldots, n\}$, we can consider the set of partitions of $\mathbf n$, which has a partial order by coarsening.  For example, we have partitions of the set $\mathbf 4$
\[ (1 2)(3)(4) < (12)(34). \] 

Thinking of this poset as a category allows us to take its classifying space and get 
a topological space.  If we include all partitions, this space is contractible, since the discrete partition consisting of singleton sets is an initial object, and the indiscrete partition consisting of the whole set is a terminal object. Discarding these two partitions results in a poset $\mathcal P(\mathbf n)$; its classifying space $|\mathcal P(\mathbf n)|$ is called a \emph{partition complex}.  

This space is of interest in a wide variety of mathematical applications, ranging from combinatorics to algebra to topology. For instance, it has been used to study the Goodwillie derivatives of the identity functor \cite{aronemah}, \cite{ching}; its homology is intimately related to Lie (super)algebras \cite{wachs}, \cite{hanlonwachs}, \cite{barcelo}, \cite{robinson:2004}; it plays a central role in the study of bar constructions for operads \cite{ching}, \cite{fresse}; and it has applications in pure combinatorics \cite[\S 7]{Stanley}.

Robinson and Whitehouse \cite{RW96,robinson:2004} first observed that the data of a partition complex can also be encoded in a suitable category of trees.  This comparison was further developed in recent work by Heuts and Moerdijk \cite{HM:21}, with an application to operad theory.  Let us briefly summarize these results; more details and formal definitions can be found in \cref{sec:P(n) and trees}.

Let $\mathcal T(\mathbf n)$ be the category of reduced $\mathbf n$-trees of \cref{ntree def}. There are two ways to show that $\mathcal P(\mathbf n)$ and $\mathcal T(\mathbf n)$ have suitably equivalent geometric realizations, as given by the following zig-zags of topological spaces:
\[
\begin{tikzcd}[column sep = scriptsize, row sep=tiny]
    {|\mathcal P (\mathbf n)| }
    &
    {|\Delta \mathcal P (\mathbf n)| }
    \arrow[r, "\simeq"]
    \arrow[l, "\simeq"']
    &
    {|\mathcal T (\mathbf n)|}
    \\
    {|\mathcal P (\mathbf n)| }
    &
    \mathbb T(\mathbf n) \arrow[l, "\cong_{\Sigma_n}"'] \arrow[r, "\cong_{\Sigma_n}"]
    &
    {|\mathcal T (\mathbf n)|.}    
\end{tikzcd}
\]
The first zig-zag uses the category of simplices $\Delta \mathcal P(\mathbf n)$ of the nerve of $\mathcal P(\mathbf n)$, and both maps are homotopy equivalences by Quillen's Theorem A. 
The argument for why the left-hand map is a homotopy equivalence can be found in \cite{Dug:06}, while the proof for the right hand map is given by Heuts and Moerdijk \cite{HM:21}.  Our work upgrades these maps to $\Sigma_n$-equivariant homotopy equivalences.  

The second zig-zag instead uses the space $\mathbb T(\mathbf n)$ of measured $\mathbf n$-trees given in \cref{def:mathbb T n}, and the maps are $\Sigma_n$-equivariant homeomorphisms.  The proof for the left-hand map was given by Robinson \cite[Theorem 2.7]{robinson:2004}, and we give an argument for the right-hand map in \cref{trees homeo}.  Notably, the second composite homeomorphism does not arise from a map between categories or simplicial sets.

Our goal is to show that these results hold in a $G$-equivariant setting, where $G$ is a finite group. As a first step, we introduce $G$-equivariant versions of the structures involved. We find that there are several possible ways to define both $G$-equivariant partition complexes and $G$-trees, depending on ``how equivariant" we ask them to be.

Given a finite set $A$, we can encode a partition of $A$ as a surjective function $A\twoheadrightarrow\mathbf k$ for some $\mathbf k$. If $A$ is now a finite $G$-set, this notion of partition still makes sense, as we can consider surjective functions on the underlying sets. Alternatively, we can ask for a non-trivial $G$-action on the target as well. That is, we can encode a partition of $A$ as a surjective function $A\to B$ where $B$ is some finite $G$-set, and either ask that the surjective map be equivariant or not. These distinctions are summarized in \cref{tab:partition cpxs}, and more details can be found in \cref{P(A) section}.

\begin{center}
{\renewcommand{\arraystretch}{1}
    \begin{table}[h!]
        \centering
        \begin{tabular}{c||c|c|c}
        & \multicolumn{3}{c}{Less equivariant \hfill More equivariant} \\
        \hline
        \hline
             \multirow{2}{*}{Partitions} & $A \twoheadrightarrow \mathbf k$ & $A \twoheadrightarrow B$ & $A \twoheadrightarrow B$  \\
             & non-equivariant & non-equivariant & equivariant\\
             \hline
             Partition complex & $\abs{\mathcal P(A)}$ & $\abs{\mathcal P_G(A)}$ & $\abs{\mathcal P^G(A)}$ \\
              \hline
        \end{tabular}
        
        \caption{Equivariant partitions}
        \label[table]{tab:partition cpxs}
    \end{table}}
\end{center}
\vspace{-1cm}

There are inclusions of poset categories $\mathcal P(A)\hookrightarrow \mathcal P_G(A)\hookleftarrow \mathcal P^G(A)$; however, it turns out that the two extreme cases, $\mathcal P(A)$ and $\mathcal P^G(A)$, have the most interesting connections to trees. The relevant types of $G$-trees are summarized in \cref{tab:equivariant trees}; see \cref{sec:G trees} for definitions and details.

\begin{center}
{\renewcommand{\arraystretch}{1}
    \begin{table}[h!]
        \centering
        \begin{tabular}{c||c|c}
        & \multicolumn{2}{c}{Less equivariant \hfill More equivariant} \\
        \hline
        \hline
             \multirow{2}{*}{Trees} & $A$-labeled trees with & $A$-labeled trees with  \\
             & $G$-action on leaves & $G$-action on entire tree\\
             \hline
             Category of trees & $\mathcal T(A)$ & $\mathcal T^G(A)$ \\
             \hline
             Space of trees & $\mathbb T(A)$ & $\mathbb T^G(A)$ \\
              \hline
        \end{tabular}
 
       \caption{Equivariant trees}
        \label[table]{tab:equivariant trees}
    \end{table}}
\end{center}
Each of these notions of $G$-trees has the expected interaction with the corresponding notion of partition; we thus obtain two different equivariant analogues of the zig-zags of equivalences above.  To prove these results, we develop equivariant versions of Quillen's Theorems A and B (\cref{thm:A,thm:B}), which we consider of independent interest.

\begin{theorem}[\cref{cor:zigzag for PA}, \cref{trees homeo}, \cref{thm:PGA and TGA}]  \label[theorem]{thm:main intro}
There are $G$-equivariant zig-zags of $G$-spaces
\[ \begin{tikzcd}[column sep = scriptsize, row sep=tiny]
    {|\mathcal P (A)| }
    &
    {|\Delta \mathcal P (A)| }
    \arrow[r, "\simeq_G"]
    \arrow[l, "\simeq_G"']
    &
    {|\mathcal T (A)|}
    \\
    {|\mathcal P (A)| }
    &
    \mathbb T(A) \arrow[l, "\cong_{G}"'] \arrow[r, "\cong_{G}"]
    &
    {|\mathcal T (A)|.}    
\end{tikzcd} \] 
By taking fixed points, there are analogous zig-zags for $\abs{\mathcal P^G(A)}$, $\mathbb{T}^G(A)$, and $\abs{\mathcal T^G(A)}$.
\end{theorem}

There are many applications of partition complexes and trees in the literature, and we can ask which of these applications have $G$-equivariant versions. We address two of these questions here.  The first is the computation of the homotopy type of a partition complex. In the non-equivariant setting, these homotopy types are given by wedges of spheres; in contrast, the situation for $G$-partition complexes is much more subtle.

\begin{theorem}[\cref{fixedptrel}, \cref{notIsovariantContractible}, \cref{EqTreesHType}]
Let $A$ be a finite $G$-set and $\downarrow^G_H \!A$ be the restriction of $A$ to an $H$-set for $H\leq G$. Then
\[ \abs{P(A)}^H \simeq \abs{P^H(\downarrow^G_H\!A)} \] 
is non-contractible only if $\downarrow^G_H\! A\cong_G \coprod_{i=1}^n H/K$ for some $K\leq H$.
\end{theorem}

This question was also addressed by Arone and Brantner \cite{AroneBrantner}, and some of our results in \cref{sec:PA homotopy} recover some of theirs, although with different proofs.  
The second question is the computation of the homology groups of spaces of trees.  In the classical setting, Robinson \cite{robinson:2004} showed that these groups are related to the Lie algebra operad via a twisted action of the integral sign representation of $\Sigma_n$.  We obtain an analogous result for our ``less equivariant'' $G$-trees by considering the integral sign representation of $G$.

\begin{theorem}[\cref{Homology of T(A)}]
There is an isomorphism of $G$-modules
\[ H^{n-3}(\mathbb T(A)) \cong \varepsilon_A^G \otimes \Lie_A, \]
where $\varepsilon_A^G$ is the sign representation of $G$ induced by the action on $A$.
\end{theorem}

In \cref{G tree homology} we explore the homology of our space of ``fully equivariant'' $G$-trees in relation to the homotopy type of their corresponding partitions.

\subsection*{Outline of the paper} 
In Section \ref{sec:P(n) and trees}, we summarize the non-equivariant comparison between partition complexes and trees, and in Section \ref{sec:eq back} we review some of the equivariant homotopy theory that we use.  We begin the equivariant story in Section \ref{sec:equivar partitions} by defining equivariant partition complexes, and we analogously define equivariant trees in Section \ref{sec:G trees}, and then in Section \ref{sec:comparison} we establish \cref{thm:main intro}.  In Section \ref{sec:PA homotopy} we discuss the homotopy type of the equivariant partition complexes, and in Section \ref{sec:homology Lie} we discuss the equivariant analogues of results relating the homology of spaces of trees to Lie algebras.

\subsection*{Acknowledgements}

This project was started at the Collaborative Workshop in Algebraic Topology in August 2022, supported by the Geometry and Topology  NSF RTG grant DMS-1839968 at University of Virginia.  We would like to thank the other participants of this workshop for an enjoyable and productive week, and the hosts at the workshop site for their hospitality. We also thank the referee for their helpful comments which improved the paper, as well as David Barnes for a helpful conversation that resulted in the addition of \cref{ex:P of C4/e}. The first-named author was partially supported by NSF grant DMS-1906281. The third-named author was supported by NSF GRFP grant DGE-1845298. The fourth-named author was partially supported by NSF grant DMS-2104300. The fifth-named author was partially supported by NSF grant DMS-2506116.

\section{A review of partition complexes and trees} \label[section]{sec:P(n) and trees}

In this section we review partition complexes, categories of trees, and the relationship between them in the non-equivariant setting.  We begin with partition complexes.  First, let us fix a finite set $\mathbf n = \{1, \ldots, n\}$ and consider the poset category $\mathcal P(\mathbf n)$ of non-trivial partitions of $\mathbf n$, ordered by coarsening, where we omit the discrete and indiscrete partitions. To turn this category into a topological space, we use the classifying space construction.  

\begin{definition} \label[definition]{defn:nerve} 
 The \emph{nerve} of a category $\mathcal C$, denoted by $N\mathcal C$, is the simplicial set whose $n$-simplices are given by functors $[n] \rightarrow \mathcal C$, where $[n]$ denotes the category with $n$ composable arrows. The \emph{classifying space} of $\mathcal C$ is the geometric realization of the nerve,
\[ \abs{\mathcal C}:= \abs{N\mathcal C}. \]
\end{definition}

\begin{definition}\label[definition]{defn:partitioncpx}
The \emph{partition complex} of $\mathbf n$ is the classifying space $\abs{\mathcal P(\mathbf n)}$ of $\mathcal P(\mathbf n)$.
\end{definition}

\begin{remark} \label[remark]{coarsening remark}
Other authors, including Heuts and Moerdijk  \cite{HM:21}, use the refinement relation on $\mathcal P(\mathbf n)$ instead.  We have chosen to use coarsening since it generalizes more conveniently to the equivariant setting in \cref{sec:equivar partitions}. Ultimately, the choice does not matter on the level of classifying spaces.
\end{remark}

\begin{definition}
For any category $\mathcal C$, the \emph{category of simplices} is the overcategory $\Delta \mathcal C := \Delta^{(-)} \downarrow N\mathcal C$. Explicitly, the objects are the $k$-simplices of the nerve, i.e. length $k$ chains of arrows in $\mathcal C$, and morphisms are generated by face and degeneracy maps.
\end{definition}

There is a functor $\Delta \mathcal C \to \mathcal C$ that sends a chain of arrows to its ultimate target, called the \emph{last vertex} functor. Using the discussion preceding Theorem 2.4 in \cite{Dug:06}, the last vertex map is homotopy initial (\cref{htpy initial def}) and hence by Quillen's Theorem A induces a homotopy equivalence on classifying spaces. It follows that $|\Delta \mathcal P(\mathbf n)|$ is another model for the partition complex.

We now introduce several varieties of trees, studied in \cite{robinson:2004} and \cite{HM:21}, that connect to the partition complex. By a \emph{tree}, we always mean a finite tree whose internal edges are attached to a vertex at both ends, but whose external edges are only attached to a single vertex. One external edge is distinguished as the \emph{root} of the tree, and the other external edges are called \emph{leaves}. The tree is oriented from the leaves down to the root.  Additionally, our trees are prohibited from having nullary vertices; see \cref{layeredtrees ex}. 

\begin{notation}
    For a tree $T$, we denote by $L(T)$, $V(T)$, and $E^i(T)$ the sets of leaves, vertices, and inner edges of $T$, respectively.
\end{notation}

\begin{definition}  \label[definition]{ntree def}
    For any $n > 0$, an \emph{$n$-labeled tree}, or simply \emph{$\mathbf{n}$-tree}, is a tree equipped with a labeling bijection $\mathbf n \to L(T)$. We say an $\mathbf{n}$-tree is
    \begin{itemize}
    \item \emph{layered} if there is a constant number of inner edges between any leaf and the root;
    
    \item \emph{reduced} if there are no unary vertices; and 
    
    \item \emph{measured} if it is equipped with the data of an assignment $E^i(T) \to (0,1]$ giving every inner edge a length in $(0,1]$, such that at least one inner edge has length 1. 
    \end{itemize}
    An \emph{isomorphism} of (reduced) $\mathbf{n}$-trees is a root-preserving homeomorphism. It is an isomorphism of labeled trees if it also preserves the labels, and an isomorphism of measured trees if it preserves edge measurements.
\end{definition}

\begin{remark}
What we call ``reduced $\mathbf{n}$-trees" are called ``$\mathbf{n}$-trees" in \cite{HM:21} and \cite{robinson:2004}.  Robinson uses the term ``fully grown $\mathbf{n}$-trees" for what we call ``measured $\mathbf{n}$-trees". 
\end{remark}

Let us look at these different kinds of trees in more depth.  First, we observe that the category of simplices $\Delta \mathcal P(\mathbf n)$ is isomorphic to the category of (isomorphism classes of) layered $\mathbf{n}$-trees, with face maps contracting an entire layer and degeneracy maps inserting a layer of unary edges.

\begin{example}  \label[example]{layeredtrees ex}
Let $\mathbf n = \mathbf 6$ and consider the $2$-simplex in $N\mathcal P(\mathbf{6})$
\[ (1)(2)(34)(5)(6) \leq (12)(34)(56) \leq (12)(3456) .\] 
This chain of partitions corresponds to the layered tree with 3 internal layers: 
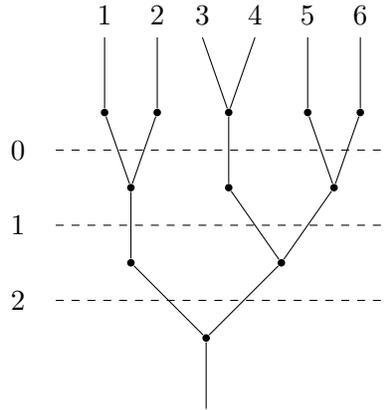
\[ \scalebox{0.85}{ \begin{tikzpicture} 
[level distance=10mm, 
every node/.style={fill, circle, minimum size=.1cm, inner sep=0pt}, 
level 1/.style={sibling distance=20mm}, 
level 2/.style={sibling distance=20mm}, 
level 3/.style={sibling distance=14mm},
level 4/.style={sibling distance=7mm}]

\node (tree)[style={color=white}] {} [grow'=up] 
child {node (level1) {} 
	child{ node {}
		child{ node {}
			child{ node (level4) {}
			child}
			child{ node (level4) {}
			child}
		}
	}
	child{ node (level2) {}
		child{ node (level3) {}
			child{ node (level4) {}
			child
			child}
		}
		child{ node {}
			child{ node (level4) {}
			child}
			child{ node (level4) {}
			child}
		}
	}
};

\tikzstyle{every node}=[]

\draw[dashed] ($(level1) + (-2cm, .5cm)$) -- ($(level1) + (2.5cm, .5cm)$);
\draw[dashed] ($(level1) + (-2cm, 1.5cm)$) -- ($(level1) + (2.5cm, 1.5cm)$);
\draw[dashed] ($(level1) + (-2cm, 2.5cm)$) -- ($(level1) + (2.5cm, 2.5cm)$);

\node at ($(level1) + (-2.5cm, 2.5cm)$) {$0$};
\node at ($(level1) + (-2.5cm, 1.5cm)$) {$1$};
\node at ($(level1) + (-2.5cm, .5cm)$) {$2$};

\node at ($(level4) + (-3.4cm, 1.3cm)$){$1$};
\node at ($(level4) + (-2.7cm, 1.3cm)$){$2$};
\node at ($(level4) + (-2.1cm, 1.3cm)$){$3$};
\node at ($(level4) + (-1.4cm, 1.3cm)$){$4$};
\node at ($(level4) + (-0.7cm, 1.3cm)$){$5$};
\node at ($(level4) + (0cm, 1.3cm)$){$6$};

\end{tikzpicture}}. \]
Here, layer 0 corresponds to $(1)(2)(34)(5)(6)$, layer 1 to $(12)(34)(56)$, and layer 2 to $(12)(3456)$. The face map $d_0$ contracts the 0-th layer, i.e. all the edges that intersect with the dashed line labeled by $0$, resulting in the tree

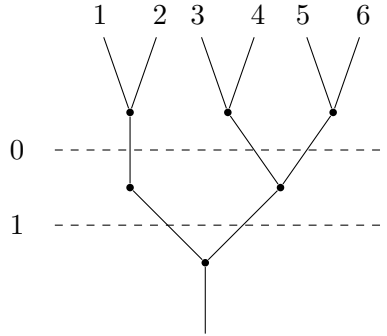
\[\scalebox{0.85}{ 
\begin{tikzpicture} 
[level distance=10mm, 
every node/.style={fill, circle, minimum size=.1cm, inner sep=0pt}, 
level 1/.style={sibling distance=20mm}, 
level 2/.style={sibling distance=20mm}, 
level 3/.style={sibling distance=14mm},
level 4/.style={sibling distance=7mm}]

\node (tree)[style={color=white}] {} [grow'=up] 
child {node (level1) {} 
	child{ node {}
		child{ node (level 3) {}
			child
			child
		}
	}
	child{ node (level2) {}
		child{ node (level3) {}
			child
			child
		}
		child{ node {}
			child
			child
		}
	}
};

\tikzstyle{every node}=[]

\draw[dashed] ($(level1) + (-2cm, .5cm)$) -- ($(level1) + (2.5cm, .5cm)$);
\draw[dashed] ($(level1) + (-2cm, 1.5cm)$) -- ($(level1) + (2.5cm, 1.5cm)$);

\node at ($(level1) + (-2.5cm, 1.5cm)$) {$0$};
\node at ($(level1) + (-2.5cm, .5cm)$) {$1$};

\node at ($(level3) + (-1.7cm, 1.3cm)$){$1$};
\node at ($(level3) + (-0.9cm, 1.3cm)$){$2$};
\node at ($(level3) + (-0.4cm, 1.3cm)$){$3$};
\node at ($(level3) + (0.4cm, 1.3cm)$){$4$};
\node at ($(level3) + (1cm, 1.3cm)$){$5$};
\node at ($(level3) + (1.8cm, 1.3cm)$){$6$};

\end{tikzpicture} } \]
that corresponds to the chain $(12)(34)(56)\leq (12)(3456)$.  We leave it to the reader to compute the other face maps as well as the degeneracy maps.
\end{example}

We say a layered tree is \emph{non-degenerate} if its associated simplex is. Visually, this condition means that there is no layer whose vertices are all unary. Additionally, a layered tree is \emph{elementary} if every layer contains exactly one non-unary vertex.  A vertex is in a layer if it is the source of an edge in the layer. Both examples above are non-degenerate, but neither is elementary.

\begin{remark}
The exclusion of the trivial partitions in $\mathcal P(\mathbf n)$ imposes restrictions on what a layered tree can look like before the first layer and after the final layer. Specifically, excluding the coarsest partition means we do not allow the layer closest to the root in any $k$-simplex to be degenerate (depicted below left), and and excluding the finest partition means we do not allow the 0-th layer to be degenerate (depicted below right):
\[ \scalebox{0.85}{ \begin{tikzpicture}
[level distance=10mm, 
every node/.style={fill, circle, minimum size=.1cm, inner sep=0pt}, 
level 1/.style={sibling distance=20mm}, 
level 2/.style={sibling distance=20mm}, 
level 3/.style={sibling distance=14mm},
level 4/.style={sibling distance=7mm}]

\node (tree)[style={color=white}] {} [grow'=up] 
child {node (level1) {} 
	child
};
\tikzstyle{every node}=[]

\draw[dashed] ($(level1) + (-2cm, .5cm)$) -- ($(level1) + (2.5cm, .5cm)$);

\node at ($(level1) + (-2.5cm, .5cm)$) {$k$};

\end{tikzpicture} }
\hspace{1.5cm}
 \scalebox{0.85}{ \begin{tikzpicture}
[level distance=10mm, 
every node/.style={fill, circle, minimum size=.1cm, inner sep=0pt}, 
level 1/.style={sibling distance=20mm}, 
level 2/.style={sibling distance=20mm}, 
level 3/.style={sibling distance=14mm},
level 4/.style={sibling distance=7mm}]

\node at (-1, 0) (tree)[style={color=white}] {} [grow'=up] 
child {node (level2) {} 
child
};

\node at (0, 0) (tree)[style={color=white}] {} [grow'=up] 
child {node (level2) {} 
child
};

\node at (1.5, 0) (tree)[style={color=white}] {} [grow'=up] 
child {node (level2) {} 
child
};

\tikzstyle{every node}=[]

\draw[dashed] ($(level1) + (-2cm, -.5cm)$) -- ($(level1) + (2.5cm, -.5cm)$);
\node at (0.75, 1) {$\dots$};

\node at ($(level1) + (-2.5cm, -0.5cm)$) {$0$};

\node at ($(level1) + (-1cm, 1.3cm)$){$1$};
\node at ($(level1) + (0cm, 1.3cm)$){$2$};
\node at ($(level1) + (0.75cm, 1.3cm)$){$\dots $};
\node at ($(level1) + (1.5cm, 1.3cm)$){$n$};
\end{tikzpicture} }.
\] 
\end{remark}

\begin{definition}
Let $\mathcal T(\mathbf n)$ denote the poset whose objects are isomorphism classes of reduced $\mathbf{n}$-trees, where there is a unique morphism $T\to T'$ if $T'$ can be obtained from $T$ by contracting a collection of inner edges.  In this case we say $T'$ is a \emph{face} of $T$, as illustrated by the following picture:
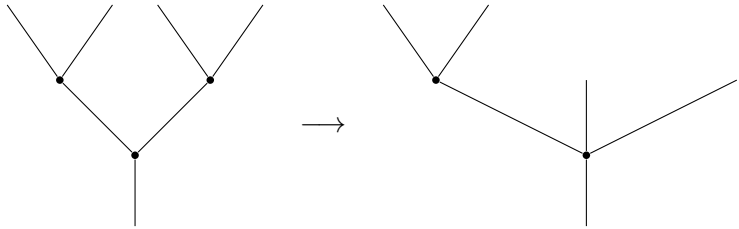
\[\scalebox{0.85}{  \begin{tikzpicture} 
[level distance=10mm, 
every node/.style={fill, circle, minimum size=.1cm, inner sep=0pt}, 
level 1/.style={sibling distance=20mm}, 
level 2/.style={sibling distance=20mm}, 
level 3/.style={sibling distance=14mm},
level 4/.style={sibling distance=7mm}]

\node (tree) at (-3,0) [style={color=white}] {} [grow'=up] 
child {node (level1) {} 
	child{ node {}
		child
		child
	}
	child{ node (level2) {}
		child
		child
	}
};

\node[fill=white, label=$\longrightarrow$] at (-0.5,1) {};

\node (tree) at (3,0) [style={color=white}] {} [grow'=up] 
child {node (level2) {} 
	child{ node (level3) {}
		child
		child
	}
	child
	child
};

\tikzstyle{every node}=[]
\end{tikzpicture}}.
\]
This category has an terminal object, the \textit{corolla} $C_{\mathbf n}$, 
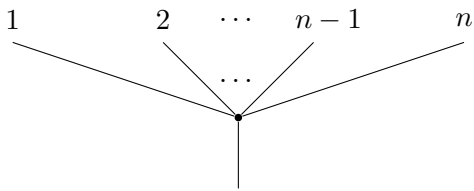
\[ \scalebox{0.85}{ \begin{tikzpicture}
[level distance=10mm, 
every node/.style={fill, circle, minimum size=.1cm, inner sep=0pt}, 
level 1/.style={sibling distance=20mm}, 
level 2/.style={sibling distance=20mm}, 
level 3/.style={sibling distance=14mm},
level 4/.style={sibling distance=7mm}]

\node (tree)[style={color=white}] {} [grow'=up] 
child {node (level1) {} 
	child
	child
	child
	child
};

\tikzstyle{every node}=[]

\node at ($(level1) + (-3cm, 1.3cm)$){$1$};
\node at ($(level1) + (-1cm, 1.3cm)$){$2$};
\node at ($(level1) + (0cm, 1.3cm)$){$\dots $};
\node at ($(level1) + (1.2cm, 1.3cm)$){$n-1$};
\node at ($(level1) + (3cm, 1.3cm)$){$n$};

\node at ($(level1) + (0cm, 0.5cm)$){$\dots$};
			\end{tikzpicture}} \] 
but for the rest of this paper we omit this object from $\mathcal{T}(\mathbf n)$.
\end{definition}

\begin{remark}
    The category $\mathcal T(\mathbf n)$ is the \emph{opposite} of the category denoted by $\mathcal T_+(n)$ in \cite{HM:21}, due to our choice of ordering $\mathcal P(\mathbf n)$ via coarsening; see \cref{coarsening remark}. Hence, these arrows are the opposite of the maps of trees in the dendroidal category $\Omega$.
\end{remark}

In \cite{HM:21}, Heuts and Moerdijk show that the functor 
\[ \begin{tikzcd}
\Delta \mathcal P(\mathbf n) \ar[r]
& \mathcal T(\mathbf n)
\end{tikzcd} \] 
that collapses unary vertices and forgets layers is homotopy final (\cref{htpy initial def}), and so again induces a homotopy equivalence $\abs{ \mathcal P(\mathbf n)} \simeq \abs{ \mathcal T(\mathbf n)}$ on classifying spaces. 

There are several rules governing the behavior of edges and leaves in a reduced tree, as well as the impact of a morphism $T\to T'$ in $\mathcal T(\mathbf n)$ on the sets $E^i(T)$ and $E^i(T')$ of inner edges. We now establish some technical results in this direction that will be useful for our goals in \cref{sec:comparison}.

\begin{definition}
    Given two edges $e$ and $f$ in a tree $T$, we say that $e\leq f$ if every path in $T$ from a point in $e$ to a point in the root must pass through $f$. Similarly, for a vertex $v$ of $T$, we say $e\leq v$ if every path in $T$ from a point in $e$ to a point in the root must pass through $v$.
\end{definition}

It is straightforward to check that this relation turns the set of edges in $T$ into a poset with maximal element given by the root.  It is often convenient also to extend this relation to the set of edges and vertices.  

For any edge $e$ we write $\Lambda(e)$ for the set of leaves $\ell\leq e$.  Similarly, for a fixed vertex $v$, we denote by $\Lambda(v)$ the set of leaves $\ell\leq v$. 

\begin{lemma}\label[lemma]{lemma: inclusion of leaf sets}
    If $e\leq f$ then there is an inclusion $\Lambda(e)\subseteq \Lambda(f)$.  If $e$ is strictly less than $f$ then this inclusion is also strict.
\end{lemma}

\begin{proof}
    The first claim follows from transitivity of the poset relation: if $\ell \leq e$ and $e\leq f$ then $\ell\leq f$.  For the second claim, suppose that $e<f$ and let $v$ be the outgoing vertex of $e$.  Since $\Lambda(e)\subseteq\Lambda(v)\subseteq \Lambda(f)$, it suffices to prove that the inclusion $\Lambda(e)\subseteq \Lambda(v)$ is strict.  Since the tree $T$ is reduced, there must be some edge $g\neq e$ which is incoming to $v$.  But then $\Lambda(g)$ is a non-empty subset of $\Lambda(v)$ which is disjoint from $\Lambda(e)$, which proves that $\Lambda(e)\neq \Lambda(v)$.   
\end{proof}

\begin{proposition}\label[proposition]{prop: edges determined by leaves}
    An edge $e$ in a reduced $\mathbf n$-tree is uniquely determined, up to label-preserving isomorphism, by the set of leaves $\ell$ such that $\ell\leq e$.
\end{proposition}

\begin{proof}
    It suffices to prove that if $e$ and $f$ are two distinct edges of a reduced $\mathbf{n}$-tree $T$, then there is some leaf $\ell$ such that $\ell\leq e$ but $\ell\nleq f$, or vice versa.  Let $v$ be the least element of the poset $V(T)$ that is greater than both $e$ and $f$.  Such a least element exists because the subset of $V(T)$ consisting of elements greater than both $e$ and $f$ is a subset of the linearly ordered finite set of vertices larger than $e$.  Let $e'$ and $f'$ be the unique incoming edges of $v$ such that $e'\geq e$ and $f'\geq f$. If $f'\neq e'$ then $\Lambda(f')\cap \Lambda(e') = \varnothing$.  
    
    If $f'=e'$ then we must have that either $e=e'$ or $f=f'$, as otherwise the vertex directly above $e'=f'$ is greater than both $e$ and $f$ and strictly less than $v$, contradicting the minimality of $v$.  Without loss of generality, assume that $e=e'$. Then we must have $f\neq e'$, since $e\neq f$, and thus $f<e$ which implies that the containment $\Lambda(f)\subseteq \Lambda(e)$ is strict by \cref{lemma: inclusion of leaf sets}.  Thus, there is some $\ell\geq e$ but $\ell\nleq f$, as claimed.
\end{proof}

\begin{corollary}\label[corollary]{cor:inclusionofinneredges}
    If $T\to T'$ is a morphism in $\mathcal T(\mathbf n)$, then there is a canonical inclusion $E^i(T')\hookrightarrow E^i(T)$.
\end{corollary}

It is perhaps worth noting that this corollary is nontrivial, as the data of a morphism $T\to T'$ is simply the fact that $T'$ can be obtained from $T$ by a contraction of edges, and does not contain the information of which edges are contracted, or in which order.  Moreover, it is important that the leaves of $T$ and the leaves of $T'$ can be canonically identified via the labeling by $\mathbf{n}$; the analogous claim for unlabeled trees is false.

\begin{proof}
    Given an edge $e\in E^i(T')$, we claim there is an edge $\tilde{e}\in E^i(T)$ such that $\Lambda(e) = \Lambda(\tilde{e})$.  Since such an $\tilde{e}$ is necessarily unique by \cref{prop: edges determined by leaves}, it defines a canonical map $E^i(T')\to E^i(T)$. To prove the claim, we make a choice of inner edges in $T$ that can be contracted to form $T'$.  Independently of this choice, there is at least one edge $\tilde{e}$ in $T$ that ``became'' $e$.  Finally, it suffices to observe that contracting inner edges does not affect the set of leaves that live over any particular edge; in particular we must have $\Lambda(\tilde{e}) = \Lambda(e)$.
\end{proof}

Finally, we discuss the space of measured $\mathbf{n}$-trees, following \cite{RW96} and \cite{robinson:2004}.

\begin{definition}\label[definition]{def:mathbb T n}
We denote the space of (isomorphism classes of) measured $\mathbf n$-trees by $\mathbb{T}(\mathbf n)$.  It is defined as the simplicial complex whose vertices are the measured $\mathbf n$-trees with exactly one inner edge. A $k$-simplex of $\mathbb T(\mathbf n)$ corresponds to the shape of a  measured tree with $k+1$ inner edges. The vertices of such a $k$-simplex are obtained by collapsing all but one inner edge that is assigned weight $1$.
\end{definition}

In particular, given a tree shape $S$, the vertices of the simplex corresponding to $S$ are in bijection with the inner edges of $S$.  More general points in a simplex consist of measured $\mathbf n$-trees of shape $S$; in other words, we assign lengths to all inner edges, and these lengths determine the barycentric coordinates of the point.  More explicitly, if $e\in S$ is an inner edge, the $e$-th barycentric coordinate of a measured tree $T$ in the simplex $S$ is the length of the edge $e$ in $T$, divided by the sum of the lengths of all internal edges of $T$. 
A proof that the space $\mathbb{T}(n)$ defined above is actually a simplicial complex can be found in \cite[Proposition 1.2]{RW96}.

Our definition agrees with the one in \cite[\S 1]{RW96} and \cite[\S 2]{robinson:2004}, used to produce an explicit 

homeomorphism $\TT(\mathbf n)\to \abs{\mathcal P(\mathbf n)}$ in \cite[Theorem 2.7]{robinson:2004}. Moreover, a similar argument shows there is a homeomorphism $\TT(\mathbf n)\to \abs{\mathcal T(\mathbf n)}$ as well; see \cref{trees homeo}. All together, we have a zig-zag of homeomorphisms
\[ \begin{tikzcd}
\abs{\mathcal P(\mathbf n)} & \mathbf \TT(\mathbf n) \ar[r]\ar[l]& \abs{\mathcal T(\mathbf n)}
\end{tikzcd} \] 
that does not appear to arise from any functors between these categories. In \cite{robinson:2004}, Robinson shows that $\TT(\mathbf n)$ is a simplicial complex with the homotopy type of a wedge of spheres and studies connections between measured $\mathbf{n}$-trees and Lie representations. 

\section{Background on equivariant homotopy theory} \label[section]{sec:eq back}

In equivariant homotopy theory, we consider familiar objects like sets, spaces, or categories, except now we give these objects the extra structure of a group action. Throughout this paper, we assume the group $G$ is finite.

This idea of equipping an object in a category $\mathcal C$ with a $G$-action is nicely encapsulated by a functor from the one-object groupoid whose morphism group is $G$. We use $BG$ to refer to both the one-object category and the classifying space of this category. 

\begin{definition}
A \emph{$\mathcal C$-object with $G$-action} is an object of $\Fun(BG, \mathcal C)$, the category of functors from $BG$ to $\mathcal C$. \emph{Equivariant morphisms}, or simply $G$-\emph{morphisms}, are natural transformations of these functors, and we often denote the resulting category by $G\mathcal C$. 
\end{definition}

We primarily consider the following examples.
\begin{itemize}
    \item For $\mathcal C=\Set$ the category of sets, a \emph{$G$-set} is a set $A$ together with a $G$-action map $G\times A\to A$ so that $e\cdot a=a$ and $(gg')\cdot a = g\cdot (g'\cdot a)$ for all $a\in A$ and $g,g'\in G$. A $G$-\emph{map} of $G$-sets is a set map $f\colon A\to A'$ so that $g\cdot f(a) = f(g\cdot a)$ for all $a\in A$ and $g\in G$. The category of $G$-sets is denoted by $G\Set$. We can also restrict to $\cat C=\Fin$, the category of finite sets, to get a category of finite $G$-sets, denoted by $G\Fin$.
    
    \item For $\mathcal C = \Top$ the category of compactly generated weak Hausdorff spaces, a \emph{$G$-space} is a space $X$ along with a continuous map $G\times X\to X$, where $G$ is given the discrete topology. A $G$-\emph{map} of $G$-spaces is a continuous map which is equivariant on underlying sets. The category of $G$-spaces is denoted by $G\Top$.
    
    \item For $\mathcal C=\Cat$ the category of small categories, a \emph{category with $G$-action} is a category $\mathcal D$ with action functors $(g\cdot)\colon \mathcal D\to \mathcal D$ for each $g\in G$ so that $(e\cdot)=\id_{\mathcal D}$ and $(g\cdot )\circ (g'\cdot ) = (gg'\cdot)$. A $G$-\emph{functor} is a functor $F\colon \mathcal D\to \mathcal D'$ so that $g\cdot F(d) = F(g\cdot d)$ and $g\cdot F(f) = F(g\cdot f)$ for all objects $d$ of $\mathcal D$, all morphisms $f$ of $\mathcal D$, and all $g\in G$. This data assembles into a category, denoted by $G\Cat$.
\end{itemize}

\begin{remark}
What we call categories with $G$-action are sometimes called \emph{strict} $G$-categories, and $G$-functors between them are called \emph{strict} $G$-functors.  Often it can be helpful to consider \emph{pseudo} $G$-categories where the $G$-actions are only associative and unital up to natural isomorphism.  In all of the examples in this paper the actions are strictly associative and unital so we do not need to make this distinction. 
\end{remark}

\subsection{Preliminaries on equivariant topological spaces}
 
We briefly review some basic ideas in the context of $G$-spaces, specifically, although the results we cite here have analogues in the setting of $G$-sets and $G$-categories.  Our exposition primarily follows \cite{may:96}.

Many non-equivariant constructions on spaces work equally well equivariantly, and additionally in the equivariant setting we have access to new structures that can be associated to subgroups $H\leq G$. 

\begin{definition}
Let $X$ be a $G$-space and $H\leq G$.
\begin{itemize}
    \item The \emph{$H$-fixed points} of $X$ are given by the space 
    \[ X^H := \{x\in X\mid h\cdot x = x \textrm{ for all }h\in H\}. \]
    
    \item The \emph{$H$-orbits} of $X$, denoted by $X/H$, is the quotient space of $X$ by the equivalence relation generated by $x\sim h\cdot x$ for all $h\in H$.
    
    \item For $x\in X$, the \emph{isotropy subgroup}, or \emph{stabilizer}, of $x$ is
    \[ G_x := \{g\in G\mid g\cdot x=x\}\leq G. \] 
\end{itemize}
Note that $x\in X^H$ precisely when $H\leq G_x$. Both $X^H$ and $X/H$ have the structure of $W_GH$-spaces, where $W_GH = N_GH/H$ is the \emph{Weyl group} of $H$ in $G$.  Here, $N_GH$ denotes the normalizer of $H$ in $G$.
\end{definition}

The functors $G\Top\to \Top$ that take a $G$-space $X$ to its $H$-fixed points and $H$-orbits are the right and left adjoints, respectively, of the functor $\Top\to G\Top$ that gives a space the trivial $G$-action. 

Given $H\leq G$, we can also consider the \emph{restriction} functor $\downarrow^G_H \colon G\Top\to H\Top$ that only remembers the $H$-action. This functor admits a left adjoint
$\uparrow_H^G \colon H\Top \to G\Top $
called \emph{induction}. 
Given an $H$-space $Y$, the induction of $Y$ is the balanced product 
\[\uparrow_H^G(Y) = G\times_H Y=G\times Y/\sim, \] 
where $\sim$ is the relation generated by $(g,h\cdot y)\sim (gh, y)$ for $g\in G$, $y\in Y$, and $h\in H$. 
If $X$ is a $G$-space, rather than just an $H$-space, then there is a $G$-homeomorphism
\[ G\times_H X \cong_G G/H\times X .
\] 

\begin{definition}
A \emph{homotopy between $G$-maps} $X\to Y$ is a homotopy $H\colon X\times I\to Y$ that is also a $G$-map, where $I$ is given the trivial $G$-action.  A $G$-map $f\colon X\to Y$ is a \emph{(weak) $G$-equivalence} if it is a (weak) equivalence upon passage to $H$-fixed points $f^H\colon X^H\to Y^H$ for each $H\leq G$.
\end{definition}

Taking $H=e$, we see that such an $f$ needs to be a homotopy equivalence of the underlying spaces. In light of the definition above, much of equivariant homotopy theory amounts to non-equivariant homotopy theory of fixed-point spaces. 

\subsection{Preliminaries on equivariant classifying spaces}

We now establish some basic facts about classifying spaces of $G$-categories. 
If $\mathcal C$ is a category with $G$-action, then its nerve $N \mathcal C$ is the same simplicial set from \cref{defn:nerve}, which now has a $G$-action given objectwise, making it a $G$-object in $\sSet$. 

This construction is functorial, in that a $G$-functor induces a $G$-map of classifying spaces.  Note that a $G$-functor also restricts to a functor on $H$-fixed points $\mathcal C^H\to \mathcal D^H$, so we also get maps on fixed points of nerves and classifying spaces. On nerves, we have that $(N\mathcal C)^H = N(\mathcal C^H)$, and the following proposition implies that taking fixed points also commutes with taking classifying spaces.

\begin{proposition} \label[proposition]{fixed pts vs realizaton}
Let $G$ be a finite group.  For any $H \leq G$ and simplicial $G$-space $X$, taking $H$-fixed points commutes with geometric realization, i.e., there is a homeomorphism $\abs{X^H} \cong \abs{X}^H$. 
\end{proposition}

\begin{proof}
Since taking fixed points for a finite group is given by a finite limit, it suffices to show that geometric realization of simplicial spaces preserves all finite limits, for which it suffices to know that it preserves the terminal object and pullbacks.  The fact that it preserves the terminal object follows from the definition, and the fact that it preserves pullbacks was established in \cite[Corollary 11.6]{may:72}.
\end{proof}

Recall that a functor $F \colon \mathcal C \to \mathcal D$ is \emph{homotopy initial} (respectively, \emph{homotopy final}) if the overcategories $F \downarrow d$ (respectively, undercategories $d \downarrow F$) are contractible for every object $d$ of $\mathcal D$.  Quillen's Theorem A \cite[\S 1]{quillen:73} shows that such a functor induces a homotopy equivalence on classifying spaces. We can generalize this notion to $G$-functors.

\begin{definition} \label[definition]{htpy initial def}
    A $G$-functor $F \colon \mathcal C \to \mathcal D$ between $G$-categories is \emph{$G$-homotopy initial} (respectively, \emph{$G$-homotopy final}) if the overcategories $F \downarrow d$ (respectively,  undercategories $d \downarrow F$) are $G_d$-contractible for every object $d$ of $\mathcal D$.
\end{definition}

Note that, in \cite{DM:16}, Dotto and Moi instead use the terminology \emph{left $G$-cofinal} rather than $G$-initial, and \emph{right $G$-cofinal} rather than $G$-final.

In \cref{app Q thm A}, we prove that the realization of a $G$-homotopy initial or final functor is a $G$-equivalence on classifying spaces, in the form of an equivariant version of Quillen's Theorem A.

The category of simplices of a $G$-category inherits a $G$-action, and the last vertex functor $\varphi$ is a $G$-functor. In \cref{cor:last vertex}, we show that this functor is $G$-homotopy initial, which has the following consequence for the partition complex.

\begin{corollary}\label[corollary]{cor:last vertex for P(n)}
    The last vertex functor $\Delta \mathcal P(\mathbf n) \to \mathcal P(\mathbf n)$ is $\Sigma_n$-homotopy initial.
\end{corollary}

\section{$G$-partition complexes} \label[section]{sec:equivar partitions}

We now introduce equivariant versions of partition complexes; that is, we develop an analogue of $\mathcal P(\mathbf{n})$, where  the finite set $\mathbf n$ is replaced with a $G$-set $A$ so that $|A|=n$. 

To figure out what we mean by a partition of a $G$-set $A$, we first note that the data of a partition of $\mathbf n$ can be encoded as the equivalence class of a surjective function $\mathbf n \twoheadrightarrow \mathbf k$, modulo the action by $\Sigma_k$ on $\mathbf k$. As an example, the partition 
\[ (12)(345)(6) \] 
can be expressed as the function $\mathbf 6\twoheadrightarrow\mathbf 3$ given by 
\[ 1,2 \mapsto 1, \ 3,4,5\mapsto 2, \ 6\mapsto 3. \] 
The role of the equivalence relation is to identify this mapping with the map 
\[ 1,2 \mapsto 2, \ 3,4,5\mapsto 1, \ 6\mapsto 3, \] 
that determines the same partition. 
From this perspective, there are several natural ways to extend this notion to account for a $G$-action:
\begin{itemize}
    \item through non-equivariant functions $A\twoheadrightarrow\mathbf k$ where $\mathbf k$ has the trivial $G$-action;
    
    \item through $G$-maps $A\twoheadrightarrow B$ where $A$ and $B$ are $G$-sets; or
    
    \item through non-equivariant functions $A\twoheadrightarrow B$ where $A$ and $B$ are $G$-sets.
\end{itemize}

We focus on the first two notions; see \cref{weirdtrees} for a discussion on why we choose to ignore the third.

\subsection{$G$-partitions} \label[subsection]{P(A) section} 

We now explore the first notion of $G$-partitions described above.

\begin{definition}
For any $G$-set $A$, let $\mathcal P(A)$ denote the $G$-poset of non-trivial partitions of $\downarrow^G_e A$, the underlying set of $A$, ordered by coarsening.
\end{definition}

Equivalently, we can describe $\mathcal P(A)$ as the category whose objects are equivalence classes of non-equivariant surjections $A \twoheadrightarrow \mathbf{k}$ modulo the action by $\Sigma_k$, and arrows $(A \twoheadrightarrow \mathbf k)\to (A \twoheadrightarrow \mathbf j)$ are factorizations 

\[\begin{tikzcd}
A\rar[two heads]\dar[two heads] & \mathbf j\\
\mathbf{k}\ar[ur,two heads]
\end{tikzcd}.\]
As in the non-equivariant case, the trivial partitions $A\to |A|$ and $A\to \mathbf 1$ are excluded. 

Note that this data is well-defined and indeed forms a poset, since by surjectivity of $A\twoheadrightarrow \mathbf k$, any two such maps $\mathbf k \twoheadrightarrow \mathbf j$ must agree, and this factorization determines a unique factorization between any two elements of the equivalence classes of the legs.  Moreover, $\mathcal{P}(A)$ is a $G$-poset, since an element $g\in G$ acts on $\mathcal{P}(A)$ by precomposition with its inverse; that is, $g$ sends $A \twoheadrightarrow \mathbf k$ to $\begin{tikzcd}[column sep = scriptsize]A \arrow[r, "g^{-1}"] & A \arrow[r, twoheadrightarrow] & \mathbf k \end{tikzcd}$.

The following result gives us information about how to relate $\mathcal P(A)$ and $\mathcal P(\mathbf n)$, as well as their categories of simplices. The proof is omitted, as it merely consists of a detailed unpacking of the definitions involved.

\begin{lemma} \label[lemma]{restrictionlemma}
Let $A$ be a $G$-set with $|A| = n$, and let $\alpha \colon G \to \Sigma_n$ denote the group homomorphism encoding the $G$-action. Then 
\[ \mathcal P(A) = \alpha^{\**}\mathcal P(\mathbf n) \quad 
{\rm and} \quad 
 \Delta \mathcal P(A) = \alpha^{\**} \Delta \mathcal P(\mathbf n). \]
\end{lemma}

Just as for $P({\mathbf n})$,  the last vertex functor for $P(A)$ is $G$-homotopy initial; see \cref{cor:last vertex}.

\begin{corollary}\label[corollary]{cor:last vertex for P(A)}
For any $G$-set $A$, the last vertex functor $\Delta \mathcal P(A) \to \mathcal P(A)$ is $G$-homotopy initial; in particular, it is a homotopy equivalence.
\end{corollary}

The second notion of equivariant partitions is as follows. 

\begin{definition} \label[definition]{mostequivariantposet}
Let $\mathcal P^G(A)$ denote the poset of non-trivial equivariant partitions of $A$, ordered by coarsening. 
\end{definition}

In other words, $\mathcal P^G(A)$ is the category whose objects are equivalence classes of $G$-surjections between $G$-sets $A \twoheadrightarrow B$ modulo the action by $\Aut_G(B)$, and whose arrows are factorizations $A \twoheadrightarrow B' \twoheadrightarrow B$ where all maps are equivariant. The trivial partitions given by $G$-isomorphisms $A\xrightarrow{\cong} B$ and the constant map $A\to \mathbf 1$ are excluded. 
As the objects of $\mathcal P^G(A)$ are $G$-maps, the natural $G$-action on $\mathcal P^G(A)$ is the trivial one.

\subsection{Interactions through fixed points}

The next theorem provides a fundamental connection between the $G$-category of partitions $\mathcal{P}(A)$ and the category of equivariant partitions $\mathcal{P}^G(A)$.  This connection is utilized in \cref{sec:PA homotopy} to reduce questions about the equivariant homotopy type of $\mathcal{P}(A)$ to the study of the homotopy type of $\mathcal{P}^G(A)$, which, in turn, allows us to leverage classical tools, like Quillen's Theorem A, to simplify certain computations.

\begin{theorem} \label[theorem]{fixedptrel} 
For any $H\leq G$ there is an equivalence of categories 
\[ \mathcal P(A)^H \simeq \mathcal P^H(\downarrow^G_H A). \]
\end{theorem} 

\begin{proof}
    To simplify notation, we leave the $\downarrow^G_H$ implicit and simply treat $A$ as an $H$-set. We begin by defining an auxiliary category $\mathcal P^H_{\ord}(A)$ whose objects are the equivalence classes of $H$-surjections $f\colon A\twoheadrightarrow B$, where $B$ is an $H$-set equipped with a total ordering.  Morphisms in this category are the same as those of $\mathcal P^H(A)$; in particular, they are not required to respect the ordering.  One can then see that the functor $\mathcal P^H_{\ord}(A)\to\mathcal P^H(A)$ that forgets the orderings is an equivalence of categories.  It remains to check that $\mathcal P^H_{\ord}(A)$ is categorically equivalent to $\mathcal P(A)^H$.
    
    Given $f\colon A\twoheadrightarrow B$ in $\mathcal P^H_{\ord}(A)$, the total ordering on $B$ determines a unique bijection $B \xrightarrow{\cong} \mathbf{k_B}$ where $k_B = |B|$.  Define a functor $F\colon \mathcal P^H_{\ord}(A)\to \mathcal P(A)^H$ that sends the class of a map $f\colon A\twoheadrightarrow B$ to the class of 
    \[ A\xrightarrow{f} B \xrightarrow{\cong} \mathbf{k_B}. \]
    Note that $F(f)$ is $H$-fixed because for any $h\in H$, the fact that $hfh^{-1}=f$ implies that $F(f)$ and $F(f)\circ h^{-1}$ are the same up to an automorphism of $\mathbf{k_B}$, namely the one determined by $h$.  Similar reasoning shows that $F$ is well-defined, since varying the representative $f\colon A\twoheadrightarrow B$ of an equivalence class by an $H$-automorphism of $B$ only changes the value of $F(f)$ by an automorphism of $\mathbf{k_B}$.
    
    If $s\colon B\twoheadrightarrow B'$ defines a morphism in $\mathcal P^H_{\ord}(A)$, we define $F(s)$ to be the unique map that fills the following square:
    \[ \begin{tikzcd}
            B \ar["s"]{r} \ar[swap, "\cong"]{d} & B' \ar["\cong"]{d}\\
            \mathbf{k_B} \ar[swap, "F(s)"]{r} & \mathbf{k_B'}.
        \end{tikzcd} \]
    We want to show that $F$ is an equivalence of categories.  Note that since $\mathcal{P}^H_{\ord}(A)$ is equivalent to a poset, its hom-sets all have size $0$ or $1$, and so the functor $F$ is faithful.
    
    First we show that $F$ is surjective on objects.  Let $f\colon A\twoheadrightarrow\mathbf k$ represent an object in $\mathcal{P}(A)^H$, which means that for each $h\in H$ there exists a (necessarily unique) bijection $\sigma_h\colon \mathbf k\to\mathbf k$ such that the following diagram commutes:
    \[ \begin{tikzcd}
        A\rar["h"]\dar["f"', twoheadrightarrow] & A\dar["f",twoheadrightarrow]\\
        \mathbf k\rar["\sigma_h"'] & \mathbf k.
    \end{tikzcd} \]
    
    Thus $\mathbf k$ is endowed with the $H$-action given by $hi=\sigma_h (i)$ for all $i\in\mathbf k$ and $h\in H$. Note that the uniqueness of $\sigma_h$ ensures that $\sigma_e=\id$ and $\sigma_{h_1}\sigma_{h_2}=\sigma_{h_1h_2}$ and we do indeed get an $H$-action. This action is defined so that $f\colon A\twoheadrightarrow \mathbf{k}$ is an $H$-map which determines an object in $\mathcal P^H_{\ord}(A)$ whose image under $F$ is equal to $f\colon A\twoheadrightarrow \mathbf{k}$.  
    
    It remains to show that $F$ is full.  Given a morphism $\varphi\colon\mathbf k \twoheadrightarrow\mathbf j$ between objects $f\colon A\twoheadrightarrow\mathbf k$ and $f'\colon A\twoheadrightarrow\mathbf j$ in $\mathcal{P}(A)^H$, consider the following diagram:
    \[ \begin{tikzcd}
    A\rar["h"]\dar["f", twoheadrightarrow]\ar[dd,"f'"',bend right=40, twoheadrightarrow] & A\dar["f", twoheadrightarrow]\ar[dd,"f'", bend left=40, twoheadrightarrow]\\
    \mathbf k\rar["\sigma_h^{\mathbf k}"]\dar["\varphi", twoheadrightarrow] & \mathbf k\dar["\varphi", twoheadrightarrow]\\
    \mathbf j\rar["\sigma_h^{\mathbf j}"] & \mathbf j.
    \end{tikzcd} \]
    The lower square commutes when precomposed with the surjection $f$, which implies that the square itself commutes and thus $\varphi$ is an $H$-map when $\mathbf{k}$ and $\mathbf{j}$ are given $H$-actions as above.  This data determines a map $\varphi'$ between the corresponding objects $f\colon A\twoheadrightarrow\mathbf k$ and $f'\colon A\twoheadrightarrow\mathbf j$ in $\mathcal P^H_{\mathrm{ord}}(A)$ with $F(\varphi')=\varphi$, and hence $F$ is full.
\end{proof}

\begin{corollary} \label[corollary]{fixedptlemma}
  For any $H\leq G$ there is an equivalence of categories 
  \[ \Delta \mathcal P(A)^H \simeq \Delta \mathcal P^H(\downarrow^G_H A). \]
\end{corollary}

\section{$G$-trees} \label[section]{sec:G trees}

Having defined several notions of equivariant partitions, we now present the corresponding notions of trees in this equivariant context.  We refer the reader back to \cref{ntree def} for the analogous non-equivariant definitions.

\begin{definition}
    For any finite $G$-set $A$, an \emph{$A$-labeled tree}, or simply \emph{$A$-tree}, is a tree equipped with a non-equivariant labeling bijection from $A$ to the leaves of $T$. We say an $A$-tree is \emph{layered, reduced}, or \emph{measured} if the underlying $|A|$-tree is.
    
    An \emph{isomorphism} of (reduced) $A$-trees is a root-preserving homeomorphism. It is an isomorphism of labeled $A$-trees if it also preserves the labels, and an isomorphism of measured $A$-trees if it preserves edge measurements.
\end{definition}

First, we observe that, as in the non-equivariant case, the category of simplices $\Delta \mathcal P(A)$ may be described as the category of (isomorphism classes of) \emph{layered $A$-trees}.

\begin{example}
    Let $G = \Sigma_6$ and $A = \mathbf 6 = \set{1,2,3,4,5,6}$. Then both trees from \cref{layeredtrees ex} are examples of layered $\mathbf 6$-trees.
\end{example}

\begin{example} \label[example]{atree ex}
    Let $G = C_4 = \set{1,i,-1,-i}$ and $A = \set{x,ix,y,-y,iy,-iy} = C_4 \amalg C_4/C_2$, with $x=-x$ and $ix=-ix$. Then
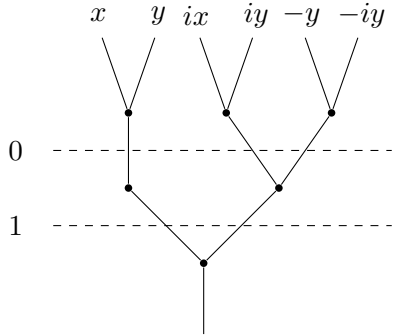
    \[\scalebox{0.85}{
    \begin{tikzpicture} 
        [level distance=10mm, 
        every node/.style={fill, circle, minimum size=.1cm, inner sep=0pt}, 
        level 1/.style={sibling distance=20mm}, 
        level 2/.style={sibling distance=20mm}, 
        level 3/.style={sibling distance=14mm},
        level 4/.style={sibling distance=7mm}]

        \node (tree)[style={color=white}] {} [grow'=up] 
        child {node (level1) {} 
	        child{ node {}
		        child{ node (level 3) {}
			        child
    			    child
        		}
	        }
    	    child{ node (level2) {}
	    	    child{ node (level3) {}
		    	    child
    			    child
        		}
	        	child{ node {}
		        	child
			        child
        		}
	        }
        };

        \tikzstyle{every node}=[]
    
        \draw[dashed] ($(level1) + (-2cm, .5cm)$) -- ($(level1) + (2.5cm, .5cm)$);
        \draw[dashed] ($(level1) + (-2cm, 1.5cm)$) -- ($(level1) + (2.5cm, 1.5cm)$);

        \node at ($(level1) + (-2.5cm, 1.5cm)$) {$0$};
        \node at ($(level1) + (-2.5cm, .5cm)$) {$1$};

        \node at ($(level3) + (-1.7cm, 1.3cm)$){$x$};
        \node at ($(level3) + (-0.9cm, 1.3cm)$){$y$};
        \node at ($(level3) + (-0.4cm, 1.3cm)$){$ix$};
        \node at ($(level3) + (0.4cm, 1.3cm)$){$iy$};
        \node at ($(level3) + (1cm, 1.3cm)$){$-y$};
        \node at ($(level3) + (1.8cm, 1.3cm)$){$-iy$};
    \end{tikzpicture}} \]
    is the layered $A$-tree corresponding to the chain of partitions
    \[ (x,y)(ix,iy)(-y,-iy) < (x,y)(ix,iy,-y,-iy). \]
    Equivalently, reading down the layers of this tree, we see that this chain corresponds to the string $A \onto \mathbf 3 \to \mathbf 2$, where $A$, $\mathbf 3$, and $\mathbf 2$ correspond to the leaves, the inner edges in layer 0, and the inner edges in layer 1, respectively. Note that the labeling of the leaves need not correspond in any way to the symmetry of the tree.    
\end{example}

As before, layered $A$-trees are defined up to label-preserving isomorphism, so, for example, we may swap the labels $ix$ and $iy$, and independently $-y$ and $-iy$ in the above example. Next, we consider the category of reduced $A$-trees.

\begin{definition}
    We denote by $\mathcal T(A)$ the category whose objects are isomorphism classes of reduced $A$-trees $T$,
    and where there is a unique morphism $T \to T'$ if $T'$ can be obtained from $T$ by contracting a collection of inner edges, and call $T'$ a \emph{face} of $T$.  As we did non-equivariantly, we omit the terminal object given by the corolla tree with no internal edges.
    
    The poset $\mathcal T(A)$ naturally has an action by $G$, where $g$ acts on objects by sending $(T, f \colon A \to L(T))$ to $(T, fg^{-1})$.
\end{definition}

\begin{example} \label[example]{atreemap ex}
    Let $G$ and $A$ be as in \cref{atree ex}.  Then there is a map in $\mathcal T(A)$
    
    \[\scalebox{0.75}{ \begin{tikzpicture} 
    [level distance=10mm, grow'=up,
    every node/.style={fill, circle, minimum size=.1cm, inner sep=0pt}, 
    level 2/.style={sibling distance=26mm}, 
    level 3/.style={sibling distance=18mm},
    level 4/.style={sibling distance=7mm}]

    \node (tree) [style={color=white}] {}
        child {node {} 
	        child{ node (xy) {}
		        child[sibling distance=8mm]
			    child[sibling distance=8mm]   		
	        }
    	    child{ node {}
	    	    child{ node (xiy) {}
		    	    child
    			    child
        		}
	        	child{ node (yiy) {}
		        	child
			        child
        		}
	        }
        };
    
    \tikzstyle{level 2}=[{sibling distance=9mm}]
    \tikzstyle{level 3}=[{sibling distance=9mm}]   
    \node (tree2) at (7,0) [style={color=white}] {}
        child{node (r) {}
            child
            child
            child{edge from parent [draw=none]}
            child[level distance=10mm]{node (s) {}
                child 
                child 
                child 
                child 
            }
        };
    
    \tikzstyle{every node}=[]

    \node[label=$\longrightarrow$] at (3.5,10mm) {};

    \node at ($(xy) + (4mm, 12mm)$){$y$};
    \node at ($(xy) + (-4mm, 12mm)$){$x$};
    
    \node at ($(xiy) + (-4.5mm, 12.5mm)$){$ix$};
    \node at ($(xiy) + (3.5mm, 12.5mm)$){$iy$};
    \node at ($(yiy) + (-4.5mm, 12.5mm)$){$-y$};
    \node at ($(yiy) + (3mm, 12.5mm)$){$-iy$};
    
    \node at ($(r) + (-14.5mm,12mm)$){$x$};
    \node at ($(r) + (-4.5mm,12mm)$){$y$};
    \node at ($(s) + (-14.5mm,12mm)$){$ix$};
    \node at ($(s) + (-4.5mm,12mm)$){$iy$};
    \node at ($(s) + (3.5mm,12mm)$){$-y$};
    \node at ($(s) + (13mm,12mm)$){$-iy$};
    \end{tikzpicture} }.\]
\end{example}

Finally, we consider the $G$-space of measured $A$-trees.

\begin{definition} \label[definition]{spaceAtrees}
We denote the $G$-space of (isomorphism classes of) measured $A$-trees by $\mathbb T(A)$. It is defined as the simplicial complex whose vertices are the measured $A$-trees with exactly one inner edge. An $n$-simplex of $\mathbb T(A)$ corresponds to a measured tree with $n+1$ inner edges whose vertices are obtained by collapsing all but one inner edge which is then assigned weight $1$. Points in such a simplex consist of measured $A$-trees of that shape; that is, they are obtained by assigning lengths to all the inner edges in that simplex shape, and in turn, these lengths determine the barycentric coordinates of the point.

The group $G$ acts on a point of $\mathbb T(A)$ by acting on the underlying $A$-labeled tree.
\end{definition}

Analogously to \cref{restrictionlemma}, we can establish the following relationship between these new notions of equivariant trees and the classical notions reviewed in \cref{sec:P(n) and trees}.

\begin{lemma}  \label[lemma]{restrictionlemmatrees}
Let $A$ be a $G$-set with $|A| = n$, and let $\alpha \colon G \to \Sigma_n$ denote the group homomorphism encoding the $G$-action. Then there is an isomorphism of $G$-categories
\[ \mathcal T(A) \cong_G \alpha^{\**} \mathcal T(\mathbf n) \] 
and a $G$-homeomorphism between spaces 
\[ \mathbb T(A) \cong_{G} \alpha^{\**} \mathbb T(\mathbf n). \] 
\end{lemma}

In order to visualize the equivariant partitions introduced in \cref{mostequivariantposet} properly, we need a corresponding more equivariant notion of $A$-tree.

\begin{definition} \label[definition]{gtree def}
    A \emph{$G$-tree} is a tree equipped with a $G$-action through root-preserving automorphisms which endows the sets of leaves, (inner) edges, and vertices with a $G$-action. An \emph{$A$-labeled $G$-tree} is a $G$-tree equipped with an equivariant labeling bijection between $A$ and the $G$-set of leaves. We say an $A$-labeled $G$-tree is 
    \begin{itemize}
        \item \emph{layered} or \emph{reduced} if the underlying $|A|$-tree is,
        
        \item \emph{$G$-elementary} if each layer has a unique $G$-orbit of vertices that are non-unary,
        
        \item \emph{$G$-measured} if the length assignment $E^i(T) \to (0,1]$ is $G$-equivariant.
    \end{itemize}
    
    An \emph{isomorphism} of (reduced) $G$-trees is a $G$-homeomorphism that preserves the root. It is an isomorphism of layered $A$-labeled $G$-trees if it also preserves labels, and an isomorphism of $G$-measured $A$-labeled $G$-trees if it preserves edge measurements. 
\end{definition}

\begin{remark}
    Note that this notion of $G$-tree is distinct from the notion with the same name in the work of the second-named author and Pereira; see \cite[\S 2.2]{BP:22}.  There, the above trees would be examples of ``trees with $G$-action", while the term $G$-tree would refer to ``orbits" of trees, say $G \cdot_H T$ for some tree $T$ with $H$-action. 
\end{remark}

As before, we associate categories and spaces to the different structures on $G$-trees. 
 
\begin{definition}
First, the category of simplices $\Delta \mathcal P^G(A)$ may be described as the category of (isomorphism classes of) \emph{layered $A$-labeled $G$-trees}, where faces and degeneracies again collapse or add layers.

Second, let $\mathcal T^G(A)$ denote the category of isomorphism classes of $A$-labeled $G$-trees, excluding the $A$-corolla. There is a unique morphism $T \to T'$ if $T'$ is obtained from $T$ by contracting a collection of inner edges; we call $T'$ a ($G$-equivariant) \emph{face} of $T$.

Third, let $\mathbb T^G(A)$ denote the space of (isomorphism classes of) measured $A$-labeled $G$-trees.
Its vertices are measured $G$-trees with exactly one orbit of inner edges. The descriptions of generic simplices and points in $\mathbb T^G(A)$ mimic the ones in \cref{spaceAtrees}. 
\end{definition}

\begin{example} \label[example]{Gtree ex}
    Let $G = \set{1,i,-1,-i}$, and $A = \set{x,ix,y,-y,iy,-iy}$ as in \cref{atree ex}.  None of the $A$-trees from \cref{atree ex} or \cref{atreemap ex} may be endowed with a $G$-action such that the $A$-labeling is $G$-equivariant. However, consider the following relabeling of the trees from \cref{atreemap ex}: 
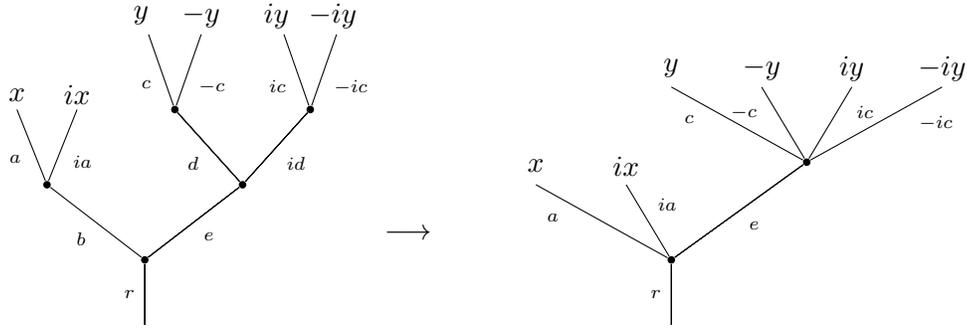
    \[\scalebox{0.85}{ \begin{tikzpicture} 
    [level distance=10mm, grow'=up, auto,
    every node/.style={font=\tiny},
    level 2/.style={sibling distance=26mm}, 
    level 3/.style={sibling distance=18mm},
    level 4/.style={sibling distance=7mm}]

    \node (tree) {}
        child {node [vertex] {} 
	        child{ node (xy) [vertex] {}
		        child[sibling distance=8mm] {edge from parent node {$a$}}
			    child[sibling distance=8mm] {edge from parent node [swap] {$ia$}}
			    edge from parent node {$b$}
	        }
    	    child{ node [vertex] {}
	    	    child{ node (xiy) [vertex] {}
		    	    child {edge from parent node {$c$}}
    			    child {edge from parent node [swap] {$-c$}}
    			    {edge from parent node {$d$}}
        		}
	        	child{ node (yiy) [vertex] {}
		        	child {edge from parent node {$ic$}}
			        child {edge from parent node [swap] {$-ic$}}
			        {edge from parent node [swap]{$id$}}
        		}
        		{edge from parent node [swap] {$e$}}
	        }
	        {edge from parent node {$r$}}
        };
    
    \tikzstyle{level 2}=[{sibling distance=12mm}]
    \tikzstyle{level 3}=[{sibling distance=12mm}]   
    \node (tree2) at (7,0) [style={color=white}] {}
        child{node (r) [vertex] {}
            child {edge from parent node [near end] {$a$}}
            child {edge from parent node [swap] {$ia$}}
            child{edge from parent [draw=none]}
            child[level distance=13mm]{node (s) [vertex] {}
                child[level distance=10mm] {edge from parent node [near end] {$c$}}
                child[level distance=10mm] {edge from parent node [very near end]{$-c$}} 
                child[level distance=10mm] {edge from parent node [swap, very near end] {$ic$}}
                child[level distance=10mm] {edge from parent node [swap, near end ]{$-ic$}}
                {edge from parent node [swap] {$e$}}
            }
            {edge from parent node {$r$}}
        };
    
    \tikzstyle{every node}=[]

    \node[label=$\longrightarrow$] at (3.5,10mm) {};

    \node at ($(xy) + (4mm, 12mm)$){$ix$};
    \node at ($(xy) + (-4mm, 12mm)$){$x$};
    
    \node at ($(xiy) + (-4.5mm, 12.5mm)$){$y$};
    \node at ($(xiy) + (3.5mm, 12.5mm)$){$-y$};
    \node at ($(yiy) + (-4.5mm, 12.5mm)$){$iy$};
    \node at ($(yiy) + (3mm, 12.5mm)$){$-iy$};
    
    \node at ($(r) + (-18mm,12.5mm)$){$x$};
    \node at ($(r) + (-6mm,12.5mm)$){$ix$};
    \node at ($(s) + (-18mm,12.5mm)$){$y$};
    \node at ($(s) + (-6mm,12.5mm)$){$-y$};
    \node at ($(s) + (6mm,12.5mm)$){$iy$};
    \node at ($(s) + (18mm,12.5mm)$){$-iy$};
    \end{tikzpicture}.} \]
    We have additionally named the edges of the tree to indicate the $G$-action.  There is an arrow between these two trees in $\mathcal T^G(A)$.
    
    However, if we had only collapsed the edge labeled by $d$ on the left, the resulting tree would not have a compatible $G$-action, and thus would not be a $G$-tree. We must collapse an entire orbit of inner edges to get a $G$-action on the quotient tree.
\end{example}

\begin{example}  \label[example]{gtreelayered ex}
    With a slight modification, the map from \cref{Gtree ex} is also a map of elementary layered $A$-labeled $G$-trees.  Consider the following trees: 
    
    \[\scalebox{0.75}
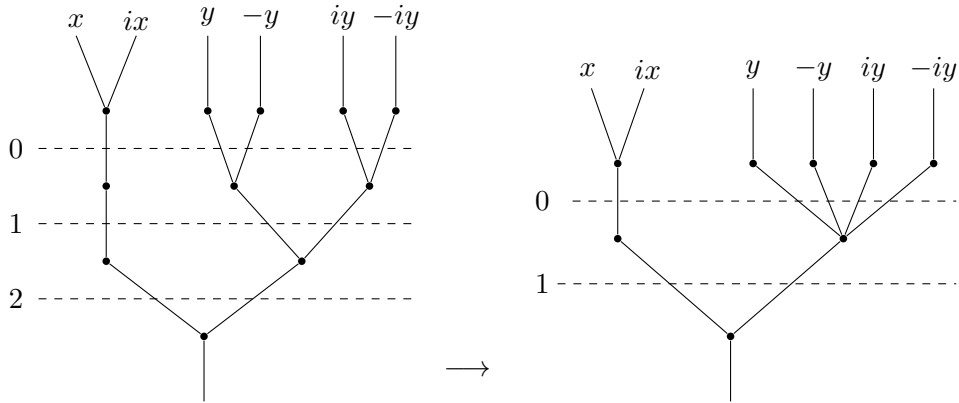
{ 
    \begin{tikzpicture} 
    [level distance=10mm, grow'=up,auto,
    every node/.style={font=\tiny},
    level 2/.style={sibling distance=26mm}, 
    level 3/.style={sibling distance=18mm},
    level 4/.style={sibling distance=7mm}]

    \node (tree) {}
        child {node [vertex] {} 
            child{ node [vertex] {}
                child{ node [vertex] {}
	                child{ node (xx) [vertex] {}
		                child[sibling distance=8mm] {}
    			        child[sibling distance=8mm] {}
	                }
	            }
	        }
    	    child{ node [vertex] {}
	    	    child{ node (xiy) [vertex] {}
		    	    child {node (y) [vertex] {}
    		    	    child{} 
		    	    }
		    	    child{ node (iy) [vertex] {}
    			        child {} 
    			    }
			    }
	        	child{ node (yiy) [vertex] {}
	        	    child{node (yy) [vertex] {}
	        	        child{}
        	        }
        	        child{ node (iyy) [vertex] {}
                        child{}
                    }
        		}
	        }
        };
    
    \tikzstyle{level 2}=[{sibling distance=30mm}]
    \tikzstyle{level 3}=[{sibling distance=8mm}]   
    \node (tree2) at (7,0) [style={color=white}] {}
        child{node (r) [vertex] {}
            child[level distance=13mm] {node [vertex] {}
                child[level distance=10mm]{ node (xxx) [vertex] {}
                    child{}
                    child{}
                }
            }
            child[level distance=13mm]{node (s) [vertex] {}
                child[level distance=10mm] {node (2y) [vertex] {} child{}}
                child[level distance=10mm] {node (2iy) [vertex] {} child{}}
                child[level distance=10mm] {node (2yy) [vertex] {} child{}}
                child[level distance=10mm] {node (2iyy) [vertex] {} child{}}
            }
        };
    
    \tikzstyle{every node}=[]
 
    \node[label=$\longrightarrow$] at (3.5,2mm) {};

    \node at ($(xx) + (4mm, 12mm)$){$ix$};
    \node at ($(xx) + (-4mm, 12mm)$){$x$};
    
    \node at ($(y) + (0, 12.5mm)$){$y$};
    \node at ($(iy) + (0, 12.5mm)$){$-y$};
    \node at ($(yy) + (0, 12.5mm)$){$iy$};
    \node at ($(iyy) + (0, 12.5mm)$){$-iy$};
    
    \node at ($(xxx) + (-4mm,12.5mm)$){$x$};
    \node at ($(xxx) + (4mm,12.5mm)$){$ix$};
    \node at ($(2y) + (0,12.5mm)$){$y$};
    \node at ($(2iy) + (0,12.5mm)$){$-y$};
    \node at ($(2yy) + (0,12.5mm)$){$iy$};
    \node at ($(2iyy) + (0,12.5mm)$){$-iy$};
    
    \draw[dashed] ($(level1) + (-2.2cm, .5cm)$) -- ($(level1) + (2.8cm, .5cm)$);
    \draw[dashed] ($(level1) + (-2.2cm, 1.5cm)$) -- ($(level1) + (2.8cm, 1.5cm)$);
    \draw[dashed] ($(level1) + (-2.2cm, 2.5cm)$) -- ($(level1) + (2.8cm, 2.5cm)$);
    
    \draw[dashed] ($(level1) + (4.7cm, .7cm)$) -- ($(level1) + (10cm, .7cm)$);
    \draw[dashed] ($(level1) + (4.9cm, 1.8cm)$) -- ($(level1) + (10cm, 1.8cm)$);

    \node at ($(level1) + (-2.5cm, 2.5cm)$) {$0$};
    \node at ($(level1) + (-2.5cm, 1.5cm)$) {$1$};
    \node at ($(level1) + (-2.5cm, 0.5cm)$) {$2$};

    \node at ($(level1) + (4.5cm, 1.8cm)$) {$0$};
    \node at ($(level1) + (4.5cm, .7cm)$) {$1$};
    \end{tikzpicture}}.
    \]
    For readability, we have dropped the names of the edges indicating the action by $G$; however, the action is just as it was previously.  Additionally, this map is between layered trees, as the arrow simply collapses the layer 1 on the left.  Finally, these trees are both $G$-elementary; in particular, even though there are two non-unary vertices in layer 1 in the tree on the left, this tree is still $G$-elementary since those two vertices are in the same $G$-orbit.
\end{example}

\begin{example}
    With $G$ and $A$ as in \cref{Gtree ex}, the tree
    \[ \scalebox{1.1}
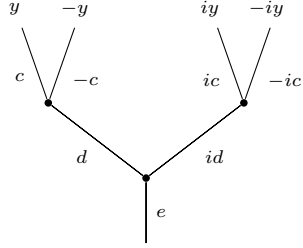
{ \begin{tikzpicture} 
    [level distance=10mm, grow'=up, auto,
    every node/.style={font=\tiny},
    level 2/.style={sibling distance=26mm}, 
    level 3/.style={sibling distance=7mm},
    level 4/.style={sibling distance=7mm}]

    \node (tree) {}
    	    child{ node [vertex] {}
	    	    child{ node (xiy) [vertex] {}
		    	    child {edge from parent node {$c$}}
    			    child {edge from parent node [swap] {$-c$}}
    			    {edge from parent node {$d$}}
        		}
	        	child{ node (yiy) [vertex] {}
		        	child {edge from parent node {$ic$}}
			        child {edge from parent node [swap] {$-ic$}}
			        {edge from parent node [swap]{$id$}}
        		}
        		{edge from parent node [swap] {$e$}}
        };
        
    \node at ($(xiy) + (-4.5mm, 12.5mm)$){$y$};
    \node at ($(xiy) + (3.5mm, 12.5mm)$){$-y$};
    \node at ($(yiy) + (-4.5mm, 12.5mm)$){$iy$};
    \node at ($(yiy) + (3mm, 12.5mm)$){$-iy$};

    \end{tikzpicture} }\]  
    is a vertex in $\mathbb T^G(A)$. However, the underlying $A$-tree is not a vertex in $\mathbb T(A)$.
\end{example}

Note that, just as with $\mathcal{P}^G(A)$, the natural $G$-actions on $\Delta \mathcal P^G(A)$, $\mathcal T^G(A)$, and $\mathbb T^G(A)$ are the trivial ones.  The different varieties of trees are strongly related, as indicated by the next lemma.

\begin{proposition} \label[proposition]{fixedptlemmatrees} 
For any $H\leq G$ there is an isomorphism of categories
\[ \mathcal T(A)^H \cong \mathcal T^H(\downarrow^G_H A) \] 
and a homeomorphism of spaces
\[ \mathbb T(A)^H \cong \mathbb T^H(\downarrow^G_H A).\] 
\end{proposition}

\begin{proof}
    We describe the homeomorphism of spaces; the isomorphism of categories is very similar with the slight wrinkle that we must consider an auxiliary category to define our functors as in the proof of \cref{fixedptrel}.  
    
    Given an $A$-labeled $H$-tree $T$, forgetting the $H$-action on $T$, but remembering the $G$-action on $A$, determines a measured $A$-tree we denote by $\varphi(T)$. Since the isomorphism class of $T$ as an $A$-labeled $H$-tree is smaller than the isomorphism class of $T$ as a measured $A$-tree, this asignment determines a well-defined continuous map
    \[ \varphi\colon \mathbb T^H(\downarrow^G_H A)\to \mathbb{T}(A). \]
    Given an $A$-labeled $H$-tree $T$, observe that the $H$-action is determined entirely by the action of $H$ on the leaves, and thus by the structure of $T$ as simply an $A$-tree. Said another way, $T$ being a $H$-tree is a property, not additional structure, which implies that $\varphi$ is injective.  Since both spaces are finite simplicial complexes, they are compact Hausdorff and so injectivity implies that $\varphi$ is a homeomorphism onto its image.

    It remains to prove $\im(\varphi) =\mathbb{T}(A)^H$.  Note that for any $A$-labeled $H$-tree $T$, the $H$-action on $T$ fixes the isomorphism class of $T$ as an $A$-tree.  Thus the image of $\varphi$ is contained in the $H$-fixed points of $\mathbb{T}(A)$.  Conversely, if a measured $A$-tree $(T, f \colon A \to L(T))$ is $H$-fixed, then for each $h \in H$, $h\cdot T = (T, fh^{-1})$ is in the same equivalence class as $T$, so there exists a tree automorphism $\sigma_h$ such that $fh^{-1} = \sigma_h f$. These $\sigma_h$ define an $H$-action on $T$ so that $f$ is an $H$-map.  If $T'$ is the resulting $A$-labeled $H$ tree, we have $\varphi(T') =T$, so we have shown $\im(\varphi) = \mathbb{T}(A)^H$.  
\end{proof}

\begin{remark} \label[remark]{weirdtrees}
The third option proposed for equivariant partitions at the beginning of \cref{sec:equivar partitions} was non-equivariant surjections $A \twoheadrightarrow B$ between $G$-sets. Using this notion in practice leads to several complications, often due to the fact that the $G$-actions and fixed points do not correspond to natural constructions.

In order to build a new $G$-poset structure $\mathcal P_G(A)$ with these objects, the arrows must be triangles so that the map $B \to B'$ is $G$-equivariant. Therefore the objects must be equivalences classes $[A \onto B]$ modulo $G$-automorphisms of $B$.  If such an equivalence class $[A \onto B]$ is $H$-fixed, the representing map need not be $H$-equivariant; instead, following the proof of \cref{fixedptrel}, the $G$-action on $B$ extends to a $G \times H$-action, and the map is $H$-equivariant with respect to the ``diagonal'' $H$-action on $B$. 

Finally, the trees that correspond to this structure are seemingly problematic, as $G$-trees equipped with a non-equivariant $A$-labeling of the leaves, modulo $G$-automorphisms of the $G$-tree. Describing the elements of an such equivalence class is a non-trivial exercise.  Once again, the $H$-fixed points correspond to $G \times H$-trees such that the $A$-labeling is $H$-equivariant with respect to the diagonal action.
\end{remark}

\section{Comparison of $G$-partition complexes and $G$-trees} \label[section]{sec:comparison}

In this section, we use the equivariant version of Quillen's Theorem A (\cref{thm:A}) to establish $G$-homotopy equivalences 
between the classifying spaces of the equivariant partition complex and several notions of equivariant trees.

To that end, let $\varphi \colon \Delta \mathcal P(\mathbf n) \to \mathcal T(\mathbf n)$ denote the functor from \cite{HM:21} that collapses unary vertices and forgets layerings.  Given a $G$-set $A$, \cref{restrictionlemma,restrictionlemmatrees} imply that this functor induces a $G$-functor
\begin{equation*} 
    \varphi \colon \Delta \mathcal P (A) \to \mathcal T(A).
\end{equation*}

\begin{theorem}\label[theorem]{partitionsastrees}
The induced $G$-functor $\varphi \colon \Delta \mathcal P (A) \to \mathcal T(A)$ is $G$-homotopy final. 
\end{theorem}

\begin{proof}
We adapt the proof in \cite{HM:21} to account for the orbital nature of $T$.

Fix a tree $T$ in $\mathcal T(A)$ and $H \leq G_T = {\rm Stab}_G(T)$. We must show that $(T \downarrow \varphi)^H$ is contractible. We first note that $T$ is an $H$-tree by \cref{fixedptlemmatrees}, and following \cref{fixedptlemma}, we define an \emph{$H$-layering of $T$} to be a layered $A$-labeled $H$-tree $S$, thought of as an object of $\Delta \mathcal P(A)^H = \Delta \mathcal P^H (\downarrow^G_H A)$, such that $\varphi(S) = T$.  Second, let $\Lambda^H(T) \subseteq N \mathcal P^H(\downarrow^G_H A)$ denote the sub-simplicial set spanned by the $H$-equivariant faces of $H$-layerings of $T$. 
Equivalently, $\Lambda^H(T)$ is generated by the elementary $H$-layerings of $T$, all of which live in simplicial degree $|V(T)/H|- 2$. 
Note that a simplex $S' \in \Lambda^H(T)$ is the face of a unique non-degenerate $H$-layering $S$ of $T$, as the face of a layering of $T$ is the layering of a unique face of $T$, and thus $S'$ induces a canonical $H$-map $T = \varphi(S) \to \varphi(S')$ in $\mathcal T^H(A)$.

We can then see that 
\[ (T \downarrow \varphi)^H = T \downarrow \varphi^H \cong \Delta \downarrow \Lambda^H(T). \]
Let $V^L(T)$ denote the $H$-set of maximal vertices of $T$, i.e. vertices whose inputs are all leaves. For any $Hv \in V^L(T)/H$, let $\Lambda^H_{Hv}(T) \subseteq \Lambda^H(T)$ denote the sub-simplicial set generated by the elementary $H$-layerings for which the vertex orbit $Hv$ is in the top layer.  Then $\Lambda^H(T) = \bigcup_{V^L(T)/H} \Lambda^H_{Hv} (T)$. But $\Lambda^H_{Hv}(T)$ is the cone on $\Lambda^H (\partial_{Hv} T)$, where $\partial_{Hv} T$ is the tree obtained from $T$ by removing all the vertices in $Hv$ and their incoming edges. Additionally, given distinct orbits $Hv_1, \dots, Hv_n$, we have that their intersection $\bigcap_{i=1, \dots, n} \Lambda^H_{Hv_i}(T)$ is the cone on $\Lambda(\partial_{Hv_1} \dots \partial_{Hv_n} T)$. Thus $\Lambda^H(T)$ is contractible, and hence so is $\Delta \downarrow \Lambda^H(T) \cong (T \downarrow \varphi)^H$.
\end{proof}

\begin{example}
    Let $G = \set{1,i,-1,-i}$ and $A = \set{x,ix,y,-y,iy,-iy}$ as in \cref{Gtree ex}. 
    Consider the tree $T$ below: 
    \[ \scalebox{0.75}
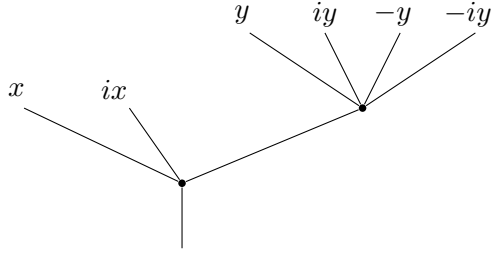
{ \begin{tikzpicture} 
        [level distance=10mm, grow'=up,auto,
        level 2/.style={sibling distance=14mm}, 
        level 3/.style={sibling distance=10mm},
        level 4/.style={sibling distance=7mm}]
        \node {}
            child{node (r) [vertex] {}
                child{}
                child{}
                child{edge from parent [draw=none]}
                child[sibling distance=16mm] {node (s) [vertex] {}
                    child{}
                    child{}
                    child{}
                    child{}
                }  
            };
        
        \node at ($(r) + (-22mm,12.5mm)$){$x$};
        \node at ($(r) + (-9mm,12.5mm)$){$ix$};
        \node at ($(s) + (-16mm,12.5mm)$){$y$};
        \node at ($(s) + (-5mm,12.5mm)$){$iy$};
        \node at ($(s) + (4mm,12.5mm)$){$-y$};
        \node at ($(s) + (14mm,12.5mm)$){$-iy$};
        \end{tikzpicture}}. \]
    Both trees from \cref{gtreelayered ex} are in $\Lambda^G(T)$: the source is an actual $G$-layering of $T$, while the target is a face.
\end{example}

\begin{remark}
    For an $H$-tree $T$ with orbital representation $T/H$, $\Lambda^H(T)$ is not equal to $\Lambda(T/H)$, as unary vertices in $T/H$ can correspond to (an orbit of) non-unary vertices in $T$. Thus, we cannot reduce the proof of \cref{partitionsastrees} to the non-equivariant case, even though the argument of the proof seems to follow as if we could.
\end{remark}

\begin{remark}
    Considering $\mathbf{n}$ with the natural $\Sigma_n$-action, this result implies that the map
    $\varphi \colon \Delta \mathcal P(\mathbf n) \to \mathcal T(\mathbf n)$ is $\Sigma_n$-homotopy final.
\end{remark}

Combining \cref{partitionsastrees,eq Q thm A,cor:last vertex} yields the following comparison.

\begin{corollary}\label[corollary]{cor:zigzag for PA}
    There is a natural zig-zag of $G$-functors 
    \[ \mathcal P(A) \xleftarrow{\ \simeq\ } \Delta \mathcal P(A) \xrightarrow{\ \simeq\ } \mathcal T(A) \] 
    that induce $G$-homotopy equivalences on classifying spaces.
\end{corollary}

As in the non-equivariant case, we have $G$-homeomorphisms between related spaces.

\begin{theorem} \label[theorem]{trees homeo}
    There are $G$-homeomorphisms 
    \[ \vert \mathcal{P}(A) \vert \cong_G \mathbb{T}(A) \cong_G |\mathcal T(A)|. \]
\end{theorem}

\begin{proof}
    The first $G$-homeomorphism follows from \cite[Theorem 2.7]{robinson:2004} and \cref{restrictionlemma,restrictionlemmatrees}, since the restriction of a $\Sigma_n$-homeomorphism is a $G$-homeomorphism.  The second follows from a $\Sigma_n$-homeomorphism $F \colon \mathbb T(\mathbf n) \to |\mathcal T(\mathbf n)|$ of a similar flavor.
    
    Given a measured $\mathbf n$-tree $T$, we get a family of $\mathbf n$-trees $S(t)$, for $0 \leq t \leq 1$, by collapsing all inner edges with lengths less than $t$ and forgetting the remaining lengths.  This family in fact produces a chain of $\mathbf n$-trees, and the barycentric coordinate of $F(T)$ with respect to $S$ is given by the amount of time $S(t) = S$.
    
    Conversely, given a (strict) chain of $\mathbf n$-trees and barycentric coordinates $(S_0 < S_1 < \dots < S_n, (\ell_0, \dots, \ell_n))$, define the measured $\mathbf n$-tree $T$ to have underlying $\mathbf n$-tree $S_0$, with the weights of $E^i(S_n)$ equal to 1, and for $0 \leq k \leq n-1$, the weights of $E^i(S_{k}) \setminus E^i(S_{k+1})$ equal to $1 - \sum_{i=k+1}^{n} \ell_i$.  Here, we are using the fact that if $T'$ is a face of $T$ then there is a canonical inclusion $E^i(T') \subseteq E^i(T)$, which is ensured by \cref{cor:inclusionofinneredges}. It is straightforward to check that these maps are continuous, $\Sigma_n$-equivariant, and inverse to one another.
\end{proof}

\begin{example}
    Consider the following element of $\mathbb T(\mathbf 6)$:
    \[ \scalebox{1}
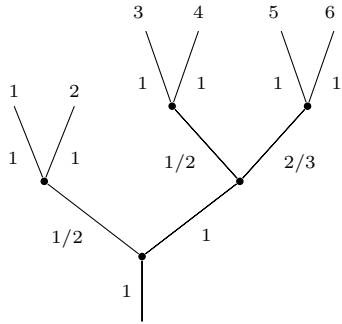
{\begin{tikzpicture} 
    [level distance=10mm, grow'=up, auto,
    every node/.style={font=\tiny},
    level 2/.style={sibling distance=26mm}, 
    level 3/.style={sibling distance=18mm},
    level 4/.style={sibling distance=7mm}]
    
    \node (tree) {}
        child {node [vertex] {} 
	        child{ node (xy) [vertex] {}
		        child[sibling distance=8mm] {edge from parent node {$1$}}
			    child[sibling distance=8mm] {edge from parent node [swap] {$1$}}
			    edge from parent node {$1/2$}
	        }
    	    child{ node [vertex] {}
	    	    child{ node (xiy) [vertex] {}
		    	    child {edge from parent node {$1$}}
    			    child {edge from parent node [swap] {$1$}}
    			    {edge from parent node {$1/2$}}
        		}
	        	child{ node (yiy) [vertex] {}
		        	child {edge from parent node {$1$}}
			        child {edge from parent node [swap] {$1$}}
			        {edge from parent node [swap]{$2/3$}}
        		}
        		{edge from parent node [swap] {$1$}}
	        }
	        {edge from parent node {$1$}}
        };
        
     \node at ($(xy) + (-4mm, 12mm)$){$1$};
    \node at ($(xy) + (4mm, 12mm)$){$2$};
    \node at ($(xiy) + (-4.5mm, 12.5mm)$){$3$};
    \node at ($(xiy) + (3.5mm, 12.5mm)$){$4$};
    \node at ($(yiy) + (-4.5mm, 12.5mm)$){$5$};
    \node at ($(yiy) + (3mm, 12.5mm)$){$6$};
    
        \end{tikzpicture}}. \]
    The map in the proof above sends this element to the $2$-simplex of $\abs{\mathcal{T}(A)}$
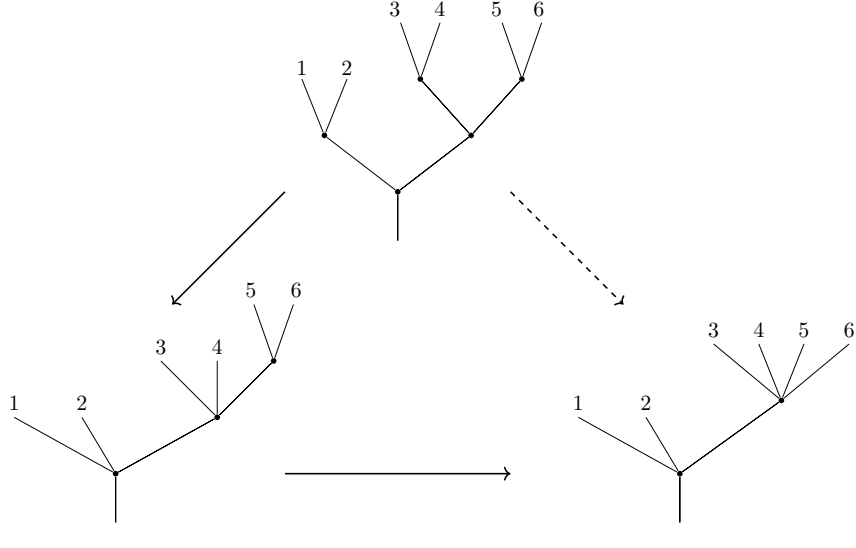
    \[ \scalebox{0.6}{
    \begin{tikzpicture} 
    [level distance=10mm, grow'=up, auto,
    every node/.style={font=\tiny},
    level 2/.style={sibling distance=26mm}, 
    level 3/.style={sibling distance=18mm},
    level 4/.style={sibling distance=7mm}]

    \node (tree) {}
        child {node [vertex] {} 
	        child{ node (xy) [vertex] {}
		        child[sibling distance=8mm] {edge from parent node {}}
			    child[sibling distance=8mm] {edge from parent node [swap] {}}
			    edge from parent node {}
	        }
    	    child{ node [vertex] {}
	    	    child{ node (xiy) [vertex] {}
		    	    child {edge from parent node {}}
    			    child {edge from parent node [swap] {}}
    			    {edge from parent node {}}
        		}
	        	child{ node (yiy) [vertex] {}
		        	child {edge from parent node {}}
			        child {edge from parent node [swap] {}}
			        {edge from parent node [swap]{}}
        		}
        		{edge from parent node [swap] {}}
	        }
	        {edge from parent node {}}
        };
    
    \tikzstyle{level 2}=[{sibling distance=12mm}]
    \tikzstyle{level 3}=[{sibling distance=8mm}]   
    \node (tree2) at (5,-5) [style={color=white}] {}
        child{node (r) [vertex] {}
            child {edge from parent node [near end] {}}
            child {edge from parent node [swap] {}}
            child{edge from parent [draw=none]}
            child[level distance=13mm]{node (s) [vertex] {}
                child[level distance=10mm] {edge from parent node [near end] {}}
                child[level distance=10mm] {edge from parent node [very near end]{}} 
                child[level distance=10mm] {edge from parent node [swap, very near end] {}}
                child[level distance=10mm] {edge from parent node [swap, near end ]{}}
                {edge from parent node [swap] {}}
            }
            {edge from parent node {}}
        };
        
    \tikzstyle{level 2}=[{sibling distance=12mm}]
    \tikzstyle{level 3}=[{sibling distance=10mm}]   
    \node (tree3) at (-5,-5) [style={color=white}] {}
	        child{ node (r') [vertex] {}
		        child {edge from parent node [near end] {}}
            child {edge from parent node [swap] {}}
            child{edge from parent [draw=none]}
    	    child{ node (m) [vertex] {}
	    	    child {edge from parent node [near end] {}}
                child {edge from parent node [swap] {}}
	        	child{ node (t) [vertex] {}
		        	child {edge from parent node {}}
			        child {edge from parent node [swap] {}}
			        {edge from parent node [swap]{}}
        		}
        		{edge from parent node [swap] {}}
	        }
	        {edge from parent node {}}
        };
    
    \tikzstyle{every node}=[]
    
     \node at ($(xy) + (-4mm, 12mm)$){$1$};
    \node at ($(xy) + (4mm, 12mm)$){$2$};
    \node at ($(xiy) + (-4.5mm, 12.5mm)$){$3$};
    \node at ($(xiy) + (3.5mm, 12.5mm)$){$4$};
    \node at ($(yiy) + (-4.5mm, 12.5mm)$){$5$};
    \node at ($(yiy) + (3mm, 12.5mm)$){$6$};
    
    \node at ($(r) + (-18mm,12.5mm)$){$1$};
    \node at ($(r) + (-6mm,12.5mm)$){$2$};
    \node at ($(s) + (-12mm,12.5mm)$){$3$};
    \node at ($(s) + (-4mm,12.5mm)$){$4$};
    \node at ($(s) + (4mm,12.5mm)$){$5$};
    \node at ($(s) + (12mm,12.5mm)$){$6$};
    
    \node at ($(r') + (-18mm,12.5mm)$){$1$};
    \node at ($(r') + (-6mm,12.5mm)$){$2$};
    \node at ($(m) + (-10mm,12.5mm)$){$3$};
    \node at ($(m) + (0mm,12.5mm)$){$4$};
    \node at ($(t) + (-4mm, 12.5mm)$){$5$};
    \node at ($(t) + (4mm, 12.5mm)$){$6$};
    
    \draw[thick, dashed, ->] (2,1)--(4,-1);
    \draw[thick, ->] (-2,1)--(-4,-1);
    \draw[thick, ->] (-2,-4)--(2,-4);
    
    \node at (0,-4.5){};
    \node at (4,0.5){};
    \node at (-4,0.5){};
    
    \end{tikzpicture}}, \] 
    with barycentric coordinates $(1/2, 1/6, 1/3)$.
\end{example}

\begin{remark}
    The composite map $|\mathcal T(\mathbf n)| \to |\mathcal P(\mathbf n)|$ is not simplicial, as it does not even send vertices to vertices.  For example, the height 3 binary tree with 4 leaves is sent to the chain $(1)(234) < (1)(2)(34)$ with barycentric coordinates $(1/2, 1/2)$.
\end{remark}

Taking fixed points yields similar results to the above comparing $G$-equivariant partitions and $G$-trees.

\begin{theorem}\label[theorem]{thm:PGA and TGA}
For any $G$-set $A$:
\begin{enumerate}[(a)]
\item The functor
\[ \varphi \colon \Delta \mathcal P^G(A) \longto \mathcal T^G(A) \]
is homotopy final, and so induces a homotopy equivalence on classifying spaces.

\item There is a natural zig-zag of functors
\[ \mathcal P^G(A) \xleftarrow{\ \simeq \ } \Delta \mathcal P^G(A) \xrightarrow{\ \simeq \ } \mathcal T^G(A) \]
that induce homotopy equivalences on classifying spaces.

\item There are homeomorphisms 
\[ |\mathcal P^G(A)| \cong \mathbb T^G(A) \cong |\mathcal T^G(A)|. \]
\end{enumerate}
\end{theorem}

\begin{proof}
    Using \cref{fixedptlemma,fixedptlemmatrees} and the fact that the fixed points of a $G$-homotopy initial (respectively, final) functor is homotopy initial (respectively, final), part (a) follows from \cref{partitionsastrees}, part (b) from (a) and \cref{cor:last vertex}, and part (c) from \cref{fixed pts vs realizaton,trees homeo}.
\end{proof}

\section{The $G$-homotopy type of $\mathcal{P}(A)$} \label[section]{sec:PA homotopy}

We now use tools developed above to study the homotopy type of the partition complexes $|\mathcal P(A)|$ and $|\mathcal P^G(A)|$. These spaces are related by \cref{fixedptrel,fixed pts vs realizaton}, which identify $|\mathcal P(A)|^H\simeq |\mathcal P^H(\downarrow^G_H A)|$ for all $H\leq G$.  As the $G$-homotopy type of $|\mathcal P(A)|$ depends on the ordinary homotopy type of its fixed points, we view computations of $|\mathcal P^H(A)|$ as stepping stones to understanding the $G$-homotopy type of~$|\mathcal P(A)|$. 

When $G=\Sigma_n$, computations of the $G$-homotopy type of $\mathcal P({\mathbf n})$ have been carried out by Arone and Brantner \cite{AroneBrantner}.  Our results are similar, but our proofs are different and make use of our explicit descriptions of the fixed point categories of $\mathcal P(A)$.

As a preview of some of the results of this section, we begin with a motivating example.

\begin{example}\label[example]{ex:P of C4/e}
    Let $G=C_4 = \{1, -1, i, -i\}$, treated multiplicatively as indicated by the names of the elements.  Let $A = C_4/e =\{1, -1, i, -i\}$ and take the left $G$-action on $A$ by left multiplication.  The action of $G$ on $|\mathcal P(A)|$ has five orbits, which partition the points of $|\mathcal P(A)|$ as in the following table.
  
\renewcommand{\arraystretch}{1.6}
\setlength{\tabcolsep}{10pt}
\begin{center}\fbox{
\begin{tabular}{c:c:c:c}

$(1,-1)(i,-1)$
& $(1)(-1,i,-i)$
& $(1)(i)(-1,-i)$
& \multirow{2}{*}{$(1)(-1)(i,-i)$} \\

\cdashline{1-1}

\multirow{2}{*}{$(1,i)(-1,-i)$} & $(i)(1,-1,-i)$ & $(-1)(i)(1,-i)$ & \\

    & $(-1)(1,i,-i)$ & $(-1)(-i)(1,i)$ & \multirow{2}{*}{$(i)(-i)(1,-1)$} \\

$(1,-i)(-1,i)$  & $(-i)(1,-1,i)$ & $(1)(-i)(-1,i)$ & \\
\end{tabular}}\end{center}
   
    Non-equivariantly, this partition complex consists of the wedge of 6 copies of $S^1$, but for our purposes, we wish to understand the $G$-structure on that space. 
    To this end, we first observe that there are four circles with trivial action that are permuted by the $G$-action, depicted below:\\

\begin{center}
    \adjustbox{width=\textwidth}{
\begin{tikzcd}[column sep = tiny, row sep = small]
	&&& {(-i)(1)(i, -1)} \\
	&& {(-i)(i,1,-1)} && {(1)(-i,i,-1)} \\
	&& {(-i)(i)(1,-1)} && {(1)(-1)(-i,i)} \\
	& {(i)(1,-1,-i)} & {(i)(-i)(1,-1)} && {(-1)(1)(-i,i)} & {(-1)(1, -i, i)} \\
	{(1)(i)(-1,-i)} &&& {(-1,1)(i,-i)} &&& {(-1)(-i)(1, i)} \\
	& {(1)(-1,i,-i)} & {(1)(-1)(i,-i)} && {(-i)(i)(-1,1)} & {(-i)(-1,1,i)} \\
	&& {(-1)(1)(i,-i)} && {(i)(-i)(-1,1)} \\
	&& {(-1)(i, -i, 1)} && {(i)(-i,-1,1)} \\
	&&& {(i)(-1)(-i,1)}
	\arrow[bend right=15, no head, thick, from=1-4, to=2-3]
	\arrow[bend left=15, no head, thick, from=1-4, to=2-5]
	\arrow[bend right=15, no head, thick, from=2-3, to=3-3]
	\arrow[bend left=15, no head, thick, from=2-5, to=3-5]
	\arrow[no head, thick, from=3-3, to=5-4]
	\arrow[no head, thick, from=3-5, to=5-4]
	\arrow[bend right = 30, no head, thick, from=4-2, to=5-1]
	\arrow[no head, thick, from=4-3, to=4-2]
	\arrow[no head, thick, from=4-3, to=5-4]
	\arrow[no head, thick, from=4-5, to=4-6]
	\arrow[bend left = 30, no head, thick, from=4-6, to=5-7]
	\arrow[no head, thick, from=5-4, to=4-5]
	\arrow[no head, thick, from=5-4, to=6-5]
	\arrow[no head, thick, from=5-4, to=7-3]
	\arrow[no head, thick, from=5-4, to=7-5]
	\arrow[bend left = 30, no head, thick, from=6-2, to=5-1]
	\arrow[no head, thick, from=6-3, to=5-4]
	\arrow[no head, thick, from=6-3, to=6-2]
	\arrow[no head, thick, from=6-5, to=6-6]
	\arrow[bend right = 30, no head, thick, from=6-6, to=5-7]
	\arrow[bend right=15, no head, thick, from=7-3, to=8-3]
	\arrow[bend left=15, no head, thick, from=7-5, to=8-5]
	\arrow[bend right = 30, no head, thick, from=8-3, to=9-4]
	\arrow[bend right = 30, no head, thick, from=9-4, to=8-5]
\end{tikzcd}.
}
\end{center}

We can then identify another circle given by the loop  depicted below:
      \[\begin{tikzpicture}
    \node (A) at ({3.3*cos(180)},{3.3*sin(180)}) {$(1,i)(-1,-i)$};
    \node (D) at ({3.3*cos(0)},{3.3*sin(0)}) {$(i)(-i)(1,-1)$};
    
    \node (B) at ({3.3*cos(120)},{2*sin(120)}) {$(-1)(-i)(1,i)$};
    \node (C) at ({3.3*cos(60)},{2*sin(60)}) {$(-i)(1,-1,i)$};
    
    \node (E) at ({3.3*cos(240)},{2*sin(240)}) {$(1)(i)(-1,-i)$};
    \node (F) at ({3.3*cos(300)},{2*sin(300)}) {$(i)(1,-1,-i)$};

    \draw (A) to[bend left=20] (B);
\draw (B) to[bend left=20] (C);
\draw (C) to[bend left=20] (D);
\draw (D) to[bend left=20] (F);
\draw (F) to[bend left=20] (E);
\draw (E) to[bend left=20] (A);
\end{tikzpicture}.\]
This circle is $G$-invariant, but not $G$-fixed, and is a copy of the representation sphere $S^\sigma$.  We leave it to the reader to identify the second such $S^\sigma$ in the diagram.

    Thus the partition complex $|\mathcal P(A)|$ is the wedge of four copies of $S^1$ with trivial $G$-action, thought of as $S^1 \wedge (C_4/e)_+$, and two representation spheres $S^\sigma$, thought of as $S^\sigma \wedge (C_4/C_2)_+$.
\end{example}

A study of the category $\mathcal P^G(A)$ reveals that its homotopy type depends heavily on the $G$-set $A$; more precisely, on whether $A$ is $H$-isovariant for some subgroup $H\leq G$, meaning there is a $G$-isomorphism $A\cong \amalg_{i=1}^n G/H$ for some $n$.
With these options in mind, we divide our approach in two cases. We show that when $A$ is not $H$-isovariant for any $H\leq G$, then the partition complex is contractible (\cref{notIsovariantContractible}) but the isovariant case is more homotopically interesting (\cref{EqTreesHType}), as \cref{ex:P of C4/e} demonstrates.

\subsection{Case 1: $A$ is not $H$-isovariant}

We first prove that if $A$ is not $H$-isovariant for any $H\leq G$, then $\mathcal P^G(A)$ is contractible.  Note that in this case $A$ must have at least two orbits, since otherwise we would have $A\cong G/G_a$ for any $a\in A$.  The first of our results only requires that $A$ have at least two orbits, so we state it in this generality.

\begin{notation}
    Let $\mathcal P_2^G(A)\subseteq \mathcal P^G(A)$ denote the full subcategory on objects $A\twoheadrightarrow B$ where $B$ has at least two $G$-orbits.
\end{notation}

\begin{lemma}\label[lemma]{contractible subcat}
    Suppose that $A$ is a non-trivial $G$-set with at least two orbits.  Then $\mathcal P_2^G(A)$ is contractible.
\end{lemma}

\begin{proof}
   Let $\mathcal C\subseteq \mathcal P_2^G(A) $ be the full subcategory on objects $f\colon A\twoheadrightarrow B$ where $B$ has trivial $G$ action.  Since $A$ is not trivial and has at least two orbits, the partition $A\twoheadrightarrow A/G$ is neither discrete nor indiscrete and thus is an object in $\mathcal C$. This object is initial in $\mathcal C$ and so $\mathcal C$ is contractible.

    Let $I\colon \mathcal C \to \mathcal P_2^G(A)$ denote the inclusion; we want to show that this functor is a homotopy equivalence.  By Quillen's Theorem A, it suffices to prove that for any object $f\colon A\twoheadrightarrow B$ in $\mathcal P_2^G(A)$, the category $f\downarrow I$ has an initial object.  By the definition of $\mathcal P^G_2(A)$, the $G$-set $B$ must have at least two orbits so $|B/G|>1$.  Let $\pi\colon B\to B/G$ denote the quotient map.  The pair $(\pi f\colon A\twoheadrightarrow B/G, B\twoheadrightarrow B/G)$ is an object in $f\downarrow I$ and is initial since any equivariant map from $A$ to a set with trivial $G$-action which factors through $B$ must also factor through $B/G$.
\end{proof}

Since almost all $G$-sets have more than one orbit, $\mathcal P^G_2(A)\subseteq \mathcal P^G(A)$ is a rather large subcategory.  We will see presently that the inclusion of this subcategory induces a homotopy equivalence whenever $A$ is not $H$-isovariant for any $H\leq G$. The argument follows Quillen's Theorem A:  if $I\colon \mathcal P^G_2(A)\to \mathcal P^G(A)$ is the inclusion, we show that the overcategory $I\downarrow (f\colon A\twoheadrightarrow B)$ is contractible for any $f\colon A\twoheadrightarrow B$ in $\mathcal P^G(A)$.  When $A = \amalg_{i=1}^nG/H$ is $H$-isovariant, our arguments show that the overcategory $I\downarrow f$ is either contractible or categorically equivalent to the partition poset $\mathcal P(\mathbf{n})$, which is never contractible.  Before proceeding, we need some notation.

\begin{definition}
 Let $H \leq G$ be a proper subgroup and let $A$ be a finite $G$-set.  We say that $A$ is \emph{$H$-induced} if there is an $G$-map $A\twoheadrightarrow G/H$. 
\end{definition}

\begin{remark}
   Let $A'$ be an $H$-set, and let $*$ denote the $H$-set with one point and trivial action.  Applying the induction functor $\uparrow_H^G\colon H\Fin\to G\Fin$ to the map $A'\to \ast$ yields a $G$-map $\uparrow_H^G(A')\to \uparrow_H^G(*)\cong G/H$.  This construction gives an equivalence of categories $H\Fin \simeq G\Fin \downarrow (G/H)$, which justifies our terminology for $H$-induced sets.  In particular, $A$ is $H$-induced if and only if there is a finite $H$-set $A'$ with $A\cong \uparrow_H^G (A')$.
\end{remark}

The following result is equivalent, by \cref{fixedptrel} above, to Lemma 6.3 in \cite{AroneBrantner} in the case where $G=\Sigma_n$. 

\begin{proposition} \label[proposition]{notIsovariantContractible}
    If $A$ is not $H$-isovariant for any $H\leq G$ then $\mathcal P^G(A)$ is contractible.
\end{proposition}

\begin{proof}
    Let $I\colon \mathcal P^G_2(A)\to \mathcal P^G(A)$ denote the inclusion of the full subcategory on objects $A \twoheadrightarrow B$ where $B$ has at least two orbits.  We want to show, under our hypotheses, that $I$ induces a homotopy equivalence so the result follows from \cref{contractible subcat}.  We prove that for any $f\colon A\twoheadrightarrow B$ in $\mathcal P^G(A)$, the undercategory $I\downarrow f$ is contractible, and hence our claim follows from (the dual of) Quillen's Theorem A. 

    If $f\colon A\twoheadrightarrow B$ is an object in $P^G_2(A)$ then $I\downarrow f$ is contractible since the identity on $f$ is a terminal object.  Suppose then that $f\colon A\twoheadrightarrow B$ is not in $\mathcal P^G_2(A)$.  Then $B$ has a single orbit and we may assume without loss of generality that $B=G/K$ for some proper subgroup $K\leq G$. In particular, $A$ is $K$-induced and so there is a finite $K$-set $A'$ so that $A\cong \uparrow_K^G (A')$.  
    
    We claim there is an equivalence of categories $I\downarrow f\simeq \mathcal P^K_2(A')$.  If so, then the fact that $A$ is not $H$-isovariant for any $H\leq G$ implies $A'$ is not $H$-isovariant for any $H\leq K$.  In particular, $A'$ has at least two orbits and is not a trivial $K$-set and so $\mathcal P^K_2(A')$ is contractible by \cref{contractible subcat}.

    An object in the category $I\downarrow f$ consists of a pair $(g\colon A\twoheadrightarrow B', h\colon B'\to G/H)$ in $G\Fin$ such that $f=hg$ and $B'$ has more than one orbit.  The claim follows from the observation that $I\downarrow f$ is equivalent to the subcategory of $(f\colon A\twoheadrightarrow G/H)\downarrow (\Set^G\downarrow (G/H))$ consisting of surjections from $A$ onto objects with at least two orbits.  Since the equivalence $H\Fin\simeq G\Fin \downarrow (G/H)$ preserves both surjections and objects with at least two orbits, we see that $I\downarrow f$ is equivalent to the subcategory of $A'\downarrow H\Fin$ with only surjections onto $H$-sets with more than one orbit, which is exactly $\mathcal P^H_2(A')$.
\end{proof}

\begin{remark}
    In the notation of the above proof, it is always true that $I\downarrow f\simeq \mathcal P^H_2(A')$.  When $A$ is $H$-isovariant, $A'$ is a trivial $H$-set and we have an equivalence of categories $\mathcal P^H_2(A')\simeq \mathcal P(|A'|)$, which is never contractible.  
\end{remark}

\subsection{Case 2: $A$ is $H$-isovariant}

We now turn our attention to studying $\mathcal P^G(A)$ when $A$ is $H$-isovariant.
The simplest case is when $A=G/H$ is a transitive $G$-set, and we can identify $\mathcal{P}^G(A)$ with an equivalent category. 

First, recall that, given $H,K\leq G$, the set of $G$-maps from $G/H$ to $G/K$ is in bijection with the set of $g\in G$ such that $gHg^{-1}\subseteq K$.  When $H\leq K$, the $G$-map $G/H\to G/K$ corresponding to $eHe^{-1}\subseteq K$ is given by $gH\mapsto gK$; we call this map the \emph{canonical quotient}. 

\begin{proposition}
     There is an equivalence of categories between the poset $\mathcal P^G(G/H)$ and the poset $S(G,H)$ of subgroups $K$ of $G$ such that $H \lneq K \lneq G$. 
\end{proposition}
    
\begin{proof}
    Define a functor $I\colon S(G,H)\to \mathcal P^G(G/H)$ that sends a subgroup $K$ to the class of the canonical quotient $G/H\twoheadrightarrow G/K$.  On morphisms, $I$ sends an inclusion of subgroups $K\leq K'$ to the canonical quotient $G/K\twoheadrightarrow G/K'$.  
    
    As the domain category is a poset, $I$ is necessarily faithful. To see that $I$ is full, note that a morphism in $\mathcal P^G(G/H)$ between canonical quotients $(G/H\twoheadrightarrow G/K)\to (G/H\twoheadrightarrow G/K')$ corresponds to a map $G/K\to G/K'$ sending $eK$ to $eK'$.  Such a map exists, and is a canonical quotient, if and only if $K\leq K'$.  
        
    Finally, we show $I$ is essentially surjective on objects. If $f\colon G/H\twoheadrightarrow B$ is surjective, then $B$ must be a transitive $G$-set, and hence
    $B\cong G/K$ where $K$ is the stabilizer of $f(eH)$.  Since the stabilizer of $eH$ is $H$, we have $H\leq K$ and $f$ is equivalent to the canonical quotient $G/H\twoheadrightarrow G/K$. 
\end{proof}

\begin{remark} \label[remark]{subgroupLatticeHType}
    The space $ |\mathcal P^G(G/H)|$ is generally non-contractible.  For example, when $G=\Sigma_3$, the space $\mathcal P^{\Sigma_3}(\Sigma_3/e)$ is equivalent to four points.  Interestingly, understanding the general homotopy type of the realization of the posets $S(G,H)$ is an open problem. When $G$ is solvable, Kratzer and Th\'evenaz show that $|S(G,e)|$ is equivalent to a wedge of equidimensional spheres \cite{KratzerThevenaz}.  However, this no longer holds for general $G$; Kramarev and Lokutsievskiy show that when $G=PSL(2,\mathbb{F}_7)$, the space $|S(G,e)|$ is homotopy equivalent to a wedge of $48$ copies of $S^1$ and $48$ copies of $S^2$ \cite{KramarevLokutsievski}. 
\end{remark}

We are left to understand the homotopy type of $\mathcal{P}^G(A)$ when $A$ is $H$-isovariant with more than one orbit. In \cref{EqTreesHType} below, we show that the homotopy type of $\mathcal P^G(A)$ for such $A$ is entirely determined by the subgroup $H\leq G$ and the number of orbits.  When $H=G$, we recover the non-equivariant partition complex, so for the remainder of the section we assume $H<G$ is a proper subgroup.

First, we fix some notation.  For any object $\alpha$ in $\mathcal P^G(A)$, let $\alpha^{\perp}$ denote the collection of objects in $\mathcal P^G(A)$ orthogonal to $\alpha$.  Thinking of $\mathcal P^G(A)$ as a poset, an element $\beta$ is in $\alpha^{\perp}$ if there is no element $\omega$ which is either a lower or upper bound for $\beta$ and $\alpha$.
    
\begin{lemma} \label[lemma]{perp lemma}
    Let $A = \amalg_{i=1}^n G/H$ for $n>1$, and let $\alpha\colon A \twoheadrightarrow \amalg_{i=1}^n G/G$ be the union of $n$ collapse maps.  Then $\alpha^{\perp} \subseteq \mathcal P^G(A)$ consists of all objects $\beta\colon A\twoheadrightarrow B$ where $B\cong G/H$.
\end{lemma}
    
\begin{proof}
    Let $\beta\colon A\twoheadrightarrow B$ represent an object in $\alpha^{\perp}$.  Then $B$ has only one orbit; otherwise, the map $A\twoheadrightarrow B\twoheadrightarrow B/G$ is an upper bound for $\alpha$ and $\beta$.  Thus $B\cong G/K$ for some subgroup $K\leq G$.  Since there is a $G$-map $A\to G/K$, $H$ must be subconjugate to $K$.
    
    If $K$ is not conjugate to $H$, the map $A\twoheadrightarrow \amalg_{i=1}^{n} G/K$ is a lower bound for $\beta$ and $\alpha$.  It follows that everything in $\alpha^{\perp}$ is of the form in the statement.  That all such objects are in $\alpha^{\perp}$ follows from similar arguments.
\end{proof}

\begin{proposition} \label[proposition]{EqTreesHType}
    There is a homotopy equivalence
    \[ |\mathcal P^G(\amalg_{i=1}^n G/H)|\simeq \bigvee\limits_{|W_G(H)|^{n-1}} |\mathcal P^G(G/H)|^{\diamond}\wedge |\mathcal P(\mathbf n)|^{\diamond} \]
    where $W_G(H) = N_G(H)/H$ is the Weyl group and $(-)^{\diamond}$ is the unreduced suspension.
\end{proposition}
    
\begin{proof}
Let $\alpha\in \mathcal P^G(A)$ be as in \cref{perp lemma}.  By \cite[3.5]{AroneBrantner} (see also \cite[4.2]{BjornerWalker}), there is a homotopy equivalence
\[ |\mathcal P^G(A)|\simeq \bigvee\limits_{\beta\in \alpha^{\perp}} |(\mathcal \beta \downarrow \mathcal P^G(A))_{\times})|^{\diamond}\wedge |(\mathcal P^G(A)\downarrow \beta)_{\times}|^{\diamond}, \]
where $\times$ denotes that we are considering the subcategory of the slice category that does not contain the initial or final objects.

By \cref{perp lemma}, we have that an arbitrary $\beta\in \alpha^{\perp}$ is of the form $A\twoheadrightarrow G/H$, and it is straightforward to check that $\mathcal (\beta \downarrow \mathcal P^G(A))_{\times}\simeq \mathcal P^G(G/H)$ and $(\mathcal P^G(A)\downarrow \beta)_{\times}\simeq \mathcal P(n)$.  It remains to check how many isomorphism classes of objects are in $\alpha^{\perp}$.  Note that every element $\beta\in \alpha^{\perp}$ is represented by an object in $\Hom_G(A,G/H)$.  Since $A$ is a disjoint union of $n$ copies of $G/H$, we have $\Hom_G(A,G/H)\cong \Aut_G(G/H)^n\cong W_G(H)^n$. Finally, we need to take the quotient by the subgroup of automorphisms of the target, which is the diagonal copy of $W_G(H)$.
\end{proof}

\begin{remark}
    The splitting of  \cref{EqTreesHType} is similar to a result of \cite{AroneBrantner}.  Let $A$ be an $H$-isovariant $G$-set and let $n=|A|$ and $|G/H|=d$. 
     The action of $G$ on $G/H$ induces an inclusion $G\subseteq \Sigma_{d}$.  The $G$-action on $A$ induces an inclusion $G\subseteq \Sigma_n$ which, up to relabeling, factors as 
    \[ G\subseteq \Sigma_d\subseteq \Sigma_d^{\frac{n}{d}}\subseteq \Sigma_n, \] 
    where the second inclusion is the diagonal embedding.  Using this embedding, \cite[Theorem 6.2]{AroneBrantner} identifies 
    \[ |\mathcal P(A)|^G \cong\uparrow_{W_{\Sigma_d}(G)\times \Sigma_{\frac{n}{d}}}^{W_{\Sigma_n}(G)} |\mathcal P^G(G/H)|^{\diamond}\wedge |\mathcal P(\mathbf n)|^{\diamond}, \]
    where $\uparrow$ is the induction functor on based spaces.  This result compares directly with \cref{EqTreesHType}, as induction on based spaces is given by wedge sum.  Counting the number of summands in both presentations, we obtain a combinatorial identity 
   \[ \frac{|W_{\Sigma_n}(G)|}{|W_{\Sigma_d}(G)|\cdot (\frac{n}{d})!} = |W_G(H)|^{n-1} \]
   that must hold whenever $G$ acts $H$-isovariantly on a set with $n$ elements.
\end{remark}

In many cases, \cref{EqTreesHType} suffices to compute the homotopy type of $|\mathcal P^G(A)|$.

\begin{corollary}\label[corollary]{HTypeOfFreeGTrees}
    If $G$ is a solvable group, $H\leq G$ is normal, and $A$ is $H$-isovariant then $|\mathcal P^G(A)|$ is homotopy equivalent to a wedge of equidimensional spheres.
\end{corollary}

\begin{proof}
    Since $H$ is normal, $Q=G/H$ is a solvable group, and we can use the equivalences of categories
    \[ \mathcal P^G(G/H)\cong S(G,H)\cong S(Q,e)\cong \mathcal P^Q(Q/e) \]
    together with \cref{subgroupLatticeHType} to deduce that $|\mathcal P^G(G/H)|$ has the homotopy type of a wedge of equidimensional spheres.   It is well-known that the homotopy type of $|\mathcal{P}(\mathbf{n})|$ is also a wedge of equidimensional spheres; see, for example \cite{robinson:2004}.   The claim now follows from \cref{EqTreesHType} and the facts that the smash product distributes over wedges and the smash product of two spheres is a sphere.
\end{proof}

\section{Connections to Cohomology and Lie algebras} \label[section]{sec:homology Lie}

Non-equivariantly, the cohomology of the space of trees is related to certain integral representations of the symmetric group $\Sigma_n$ coming from Lie algebra theory.  In this section we recall this result, following Robinson \cite{robinson:2004}, and explain how our work relates to it.  All cohomology groups in this section are integral.

Before proceeding, some remarks are in order regarding the way our work fits into the general context of equivariant cohomology theories.  For a $G$-space $X$, there are three standard ways that the action of $G$ induces additional structure on homology. The most straightforward, and the one we focus on, is that for all $g\in G$, the maps $g\colon X\to X$ induce a $G$-action on $H^*(X)$ giving it the structure of a graded $G$-module.  Two other common approaches are Borel cohomology and Bredon cohomology \cite{may:96}, but we do not consider these notions here. Computations of the Bredon homology of partition complexes for $G=\Sigma_n$ are done in work of Arone, Dwyer, and Lesh \cite{AroneDwyerLesh1}, \cite{AroneDwyerLesh2}. 

We now recall the work of Robinson on computations of the $\Sigma_n$-module structure on the cohomology of the ordinary partition complex $\mathcal P(\mathbf{n})$.  For a fixed $n$, write $\mathcal L_n$ for the free Lie algebra on a set of $n$ generators $\{x_1,\dots,x_n\}$. 
The \emph{n-linear part} of $\mathcal L_n$ is the subgroup $\Lie_n \leq \mathcal L_n$ generated by Lie monomials containing every generator $x_i$ 
exactly once.  The standard left action of the symmetric group $\Sigma_n$ on the set $\{x_1,\dots,x_n\}$ 
extends to an action on $\Lie_n$ that we call the \emph{integral Lie representation} of $\Sigma_n$.  The collection of $\mathbb{Z}[\Sigma_n]$-modules $\{\Lie_n\}$ forms a symmetric operad in abelian groups whose algebras are Lie algebras.

Let $\varepsilon^{\Sigma_n}$ denote the integral sign representation of $\Sigma_n$. The following theorem is proved in \cite[Theorem 4.1]{robinson:2004}.

\begin{theorem}\label[theorem]{Robinson Tree Homology}
    There is an isomorphism of $\Sigma_n$-modules \[H^{n-3}(\mathbb T(\mathbf{n}))\cong \varepsilon^{\Sigma_n}\otimes \Lie_n.\] 
\end{theorem}

We would like to prove an analogous result when $\mathbf{n}$ is replaced by a $G$-set $A$ for some finite group $G$.  The first step is to find suitable replacements for the $\Sigma_n$-representations ${\Lie}_n$ and $\varepsilon^{\Sigma_n}$. Given a $G$-set $A$, let $\alpha\colon G\to \Sigma_n$ be the homomorphism that realizes the action of $G$ on $A$.  Implicitly, this homomorphism depends on a choice of total ordering for $A$, but we do not use this additional information.

Let $\mathcal{L}_A$ denote the free Lie algebra on the set $A$.  Since $\mathcal{L}_A$ is generated as a Lie algebra by a set in bijection with $A$, it inherits a natural $G$-action.   We define the $A$-linear part of $\mathcal{L}_A$ to be the $G$-subgroup ${\Lie}_A\leq \mathcal{L}_A$ generated by Lie monomials containing every generator of $\mathcal{L}_A$ exactly once.  This $G$-submodule plays the role of ${\Lie}_n$ in the equivariant setting.  

To replace the sign representation, we define the $A$-sign representation $\varepsilon^G_A$ of $G$.  Let $G$ act on the free abelian group $V$ generated by $A$.  A choice of ordering for $A$ corresponds to a choice of ordered basis for $V$, and thus gives matrix representations for the action of each $g\in G$.  Define $\varepsilon^G_A(g) = \det(g) =\pm 1$ for all $g\in G$, and consider this action as a $1$-dimensional $G$-representation.  Note that while this definition requires a choice of ordering for $A$, the $G$-representation $\varepsilon^G_A$ is independent of this choice, since any two choices of ordering yield actions on $V$ that are conjugate.

It is not hard to show there are isomorphisms of $G$-modules 
\[{\Lie}_A\cong \alpha^*({\Lie}_n) \quad \textrm{and}\quad  \varepsilon^G_A\cong \alpha^*(\varepsilon^{\Sigma_n}), \] 
where $\alpha^*$ is the functor that restricts a $\Sigma_n$-module to a $G$-module along the homomorphism $\alpha\colon G\to \Sigma_n$.  The next proposition uses these isomorphisms to give an equivariant analogue of \cref{Robinson Tree Homology}.

\begin{theorem}\label[theorem]{Homology of T(A)}
    There is an isomorphism of $G$-modules
    \[ H^{n-3}(\mathbb T(A)) \cong \varepsilon_A^G \otimes \Lie_A. \]
\end{theorem}

\begin{proof}
    Unwinding the definition, we see there are isomorphisms of $G$-modules
    \[\alpha^*(\varepsilon_A^{\Sigma_n}\otimes \Lie_n)\cong \alpha^*(\varepsilon^{\Sigma_n})\otimes\alpha^*({\Lie}_n) \cong \varepsilon_A^G \otimes \Lie_A. \]
    Since $\varepsilon^{\Sigma_n} \otimes \Lie_n\cong H^{n-3}(\mathbb T(\mathbf{n}))$, the result now follows from \cref{restrictionlemmatrees} and the fact that for any $\Sigma_n$-space $Y$ there is an isomorphism of $G$-modules $H^*(\alpha^*Y)\cong \alpha^* H^*(Y)$. This last isomorphism follows from an isomorphism at the level of singular cochains.
\end{proof}

We conclude this section with some comments on the cohomology of the space of equivariant trees $\mathbb{T}^G(A)$. We rely on the homeomorphism $\mathbb{T}^G(A)\cong |\mathcal{P}^G(A)|$ from \cref{thm:PGA and TGA}. Using \cref{notIsovariantContractible}, the homology of this space is often trivial.

\begin{proposition}\label[proposition]{G tree homology}
    Suppose $A$ is a $G$-set that is not $H$-isovariant for any $H\leq G$.  Then the reduced homology $\widetilde{H}^*(\mathbb{T}^G(A))=0$. 
\end{proposition}

When $A$ is $H$-isovariant, its homology is non-trivial and, by  \cref{EqTreesHType}, it is determined by the homotopy types of $|\mathcal{P}^G(G/H)|$ and $|\mathcal{P(|A|)}|$. As noted in \cref{subgroupLatticeHType}, the homotopy type of $|\mathcal{P}^G(G/H)|$ is not known in general, and so we are unable to compute the homology of these spaces completely.  In nice cases, we can use \cref{HTypeOfFreeGTrees} to compute the cohomology when the $G$-set $A$ is free.

\begin{corollary}
    If $G$ is a solvable group and $A$ is $H$-isovariant for some $H\trianglelefteq G$, then the reduced cohomology of $\mathbb{T}^G(A)$ is a finitely generated free abelian group concentrated in a single dimension.
\end{corollary}

\appendix

\section{Equivariant versions of Theorems A and B} \label{app Q thm A} 

Quillen's Theorems A and B \cite[\S 1]{quillen:73} play central roles in Quillen's work on algebraic $K$-theory, but are also widely applied outside of that context. In this appendix, we prove analogous theorems for $G$-functors between categories with $G$-action.  While a proof of an equivariant Theorem A appeared in \cite[Theorem 3.10]{RafaelVF}, and more general versions of equivariant Theorems A and B are proved in \cite[Theorem 2.25]{DM:16} and \cite[\S 3]{Dotto}, we include the proofs here because they are short, and keep the paper self-contained.

\subsection{Equivariant Theorem A} 

Quillen's Theorem A is a useful tool for determining whether a functor induces a homotopy equivalence on classifying spaces. In particular, Theorem A is used by Heuts and Moerdijk in their comparison of partition complexes and trees \cite{HM:21}. We now prove a suitable equivariant analogue. In the special case of posets, an equivariant version of Theorem A was proved by Th\'evanaz and Webb \cite{Thevanaz}. 

The idea behind Quillen's Theorem A is that we can determine whether a functor $F\colon \mathcal C\to \mathcal D$ is a homotopy equivalence by looking at the classifying spaces of the undercategory $d\!\downarrow\!F$ for all objects $d$ of $\mathcal D$. Recall that the objects of $d\!\downarrow\! F$ are pairs $(c, g\colon d\to Fc)$, where $c$ is an object of $\mathcal C$ and $g$ is a morphism in $\mathcal D$, and that a morphism between $(c,g)$ and $(c',g')$ in $d\!\downarrow\! F$ is a map $f\colon c\to c'$ such that $g'=(Ff)g$. 

The original statement of Quillen's Theorem A is as follows.

\begin{theorem} \label[theorem]{thm:A}
Let $F \colon \mathcal C \rightarrow \mathcal D$ be a functor.  If $d\!\downarrow\!F$ is contractible for every object $d$ of $\mathcal D$, then $F$ induces a homotopy equivalence $\abs{F}\colon \abs{\mathcal C}\to \abs{\mathcal D}$.
\end{theorem}

As noted by Quillen \cite{quillen:73}, the dual statement where we assume $F$ is homotopy initial, rather than homotopy final, also holds by an analogous proof.

We want an equivariant analogue of this theorem. Observe that if $F\colon \mathcal C \to \mathcal D$ is a $G$-functor between categories with a $G$-action, then the fiber $d\!\downarrow\! F$ has an action of the isotropy subgroup $G_d=\{g\in G\mid g\cdot d = d\}$. For any $H\leq G_d$ we can compute the fixed point category $(d\downarrow F)^H$.  Note that we have an equality of categories
\[ (d\downarrow F)^H = d\downarrow F^H, \]
since both have objects $(c,\psi\colon d\to Fc)$ with $hc=c$ and $h\psi=\psi$ for all $h\in H$.

We now want to ask that each fiber $d\!\downarrow\! F$ is \emph{$G_d$-contractible}, meaning that the homotopy equivalence $\abs{d\!\downarrow\! F}\to \ast$ restricts to homotopy equivalences of fixed points $\abs{d\!\downarrow\! F}^H\to \ast$ for all $H\leq G_d$; i.e.,  $F$ must be $G$-homotopy final (see \cref{htpy initial def}). Setting $H=e$, we see the fibers all need to be contractible, as in the non-equivariant version.
We can thus state the following equivariant version of Quillen's Theorem A.

\begin{theorem}[Equivariant Theorem A]\label[theorem]{eq Q thm A}
If $F\colon \mathcal C\to \mathcal D$ is $G$-homotopy initial or $G$-homotopy final, then $\abs{F}\colon \abs{\mathcal C}\to \abs{\mathcal D}$ is a $G$-homotopy equivalence.
\end{theorem}

\begin{proof}
We focus on the case where $F$ is $G$-homotopy final, as the dual result follows by replacing the use of (non-equivariant) Theorem A with its dual theorem. 

To conclude $\abs{F}$ is a $G$-homotopy equivalence, we need to show that $\abs{F}^H\colon \abs{\mathcal C}^H\to \abs{\mathcal D}^H$ is a homotopy equivalence for all $H\leq G$. Since taking fixed points commutes with classifying spaces by \cref{fixed pts vs realizaton}, we equivalently show that $\abs{F^H}\colon \abs{\mathcal C^H}\to \abs{\mathcal D^H}$ is a homotopy equivalence by applying (non-equivariant) Theorem A. Note that if $d\in \ob\mathcal D^H$, then we must have $H\leq G_d$, and \[\abs{d\!\downarrow\!F}^H=\abs{(d\!\downarrow\!F)^H}=\abs{d\!\downarrow\! (F^H)}.\] By assumption, $d\!\downarrow\! F^H = (d\!\downarrow\! F)^H$ is contractible, so we apply Theorem A to conclude $\abs{F^H}\colon \abs{\mathcal C^H}\to \abs{\mathcal D^H}$ is a homotopy equivalence, completing the proof.
\end{proof}

Finally, we include the consequence of \cref{eq Q thm A} that we use in this paper.

\begin{corollary}\label[corollary]{cor:last vertex}
For any $G$-category $\mathcal C$, the last vertex functor $ F\colon \Delta \mathcal C \to \mathcal C$ is $G$-homotopy initial and hence induces a $G$-equivalence on classifying spaces.
\end{corollary}

\begin{proof}
    From the definitions, one can check that $F \downarrow d = \Delta(\mathcal C \downarrow d)$ for any object $d$ of $\mathcal C$, and so $(F \downarrow d)^H = F^H \downarrow d = \Delta(\mathcal C^H \downarrow d)$ for all $H \leq G_d$.  As noted in, for example, the discussion in \cite{Dug:06} before Theorem 2.4, the (non-equivariant) last vertex map induces a weak equivalence on nerves. Hence $(F \downarrow d)^H$ and $\mathcal C^H \downarrow d$ have equivalent nerves, and thus $(F \downarrow d)^H$ is contractible, as desired.
\end{proof}

\subsection{Equivariant Theorem B}

We may similarly prove an equivariant version of Quillen's Theorem B, which gives a sufficient condition to model the homotopy fiber of $\abs{F}\colon \abs{\mathcal C}\to \abs{\mathcal{} D}$ as a classifying space. The original statement of Quillen's Theorem B is as follows.

\begin{theorem}
Let $F\colon \mathcal C\to \mathcal D$ be a functor and suppose that for every morphism $d\to d'$ in $\mathcal D$, the induced map $\abs{d'\!\downarrow\!F}\to \abs{d\!\downarrow\!F}$ is a homotopy equivalence. Then the following pullback square is a homotopy pullback:
\[ \begin{tikzcd}
\abs{d\!\downarrow\!F} \ar[r]\ar[d, swap,] & \abs{\mathcal C}\ar[d, "F"] \\
\abs{d\!\downarrow\!\id_{\mathcal D}} \ar[r, swap,] & \abs{\mathcal D}.
\end{tikzcd}
\] 
\end{theorem}

Since $\id_d$ is an initial object, $\abs{d\!\downarrow\!\id_{\mathcal D}}$ is contractible, so the inclusion $\abs{d\!\downarrow\!F}\to {\rm hfib}(\abs{F})$ is a homotopy equivalence. There is also a dual version of Theorem B. 

To prove an equivariant version of Theorem B, we need the following result.

\begin{lemma} \label[lemma]{lem:hfib of G-functor}
    Suppose $F\colon \mathcal C\to \mathcal D$ is a $G$-functor and let $d\in \ob(\mathcal D)$. Then the homotopy fiber ${\rm hfib}_d(\abs{F})$ is a $G_d$-space, and  for every $H\leq G_d$ we have 
    \[ {\rm hfib}_d(\abs{F})^H\simeq {\rm hfib}_d(\abs{F^H}).
    \] 
\end{lemma}

\begin{proof}
We can model the homotopy fiber ${\rm hfib}_d(\abs{F})$ as
\[ \{d\}\times^h_{\abs{\mathcal D}}\abs{\mathcal C}= \{(d, \gamma, c)\mid c\in \abs{\mathcal C}, \gamma\in \abs{D}^I, \gamma(0)=d, \gamma(1)=\abs{F}(c)\}. \] 
This space has a $G_d$-action $g\cdot (d,\gamma, c) = (d, g\cdot \gamma, g\cdot c)$ where $(g\cdot\gamma)(t) = g\cdot \gamma(t)$ for all $t\in I$. A point $(d,\gamma,c)\in {\rm hfib}_d(\abs{F})$ is thus $H$-fixed if and only if $c\in \abs{\mathcal C}^H=\abs{\mathcal C^H}$ and $\gamma$ is a path in $\abs{\mathcal D}^H=\abs{\mathcal D^H}$, which is to say $(d,\gamma,c)\in {\rm hfib}_d(\abs{F^H})$. Thus ${\rm hfib}_d(\abs{F})^H$ models ${\rm hfib}_d(\abs{F^H})$ for $H\leq G_d$.
\end{proof}

\begin{theorem}[Equivariant Theorem B]\label[theorem]{thm:B}
Suppose $F\colon \mathcal C\to \mathcal D$ is a $G$-functor and every morphism $d\to d'$ in $\mathcal D$ induces a $G_d\cap G_{d'}$-equivalence $\abs{d'\!\downarrow\! F}\to \abs{d\!\downarrow\!F}$. Then for each object $d$ of $\mathcal D$, the inclusion $\abs{d\!\downarrow\!F}\to {\rm hfib}_d(\abs{F})$ is a $G_d$-equivalence.
\end{theorem}

\begin{proof}
For an object $d$ of $\mathcal D$, we want to show that $\abs{d\!\downarrow\!F}^H\to {\rm hfib}_d(\abs{F})^H$ is a homotopy equivalence for each $H\leq G_d$ by applying non-equivariant Theorem B to $F^H\colon \mathcal C^H\to \mathcal D^H$. Note that $d$ is always an object of $\mathcal D^H$ and \[\abs{d\!\downarrow\!F}^H=\abs{(d\!\downarrow\!F)^H}=\abs{d\!\downarrow\! (F^H)}.\]

In order to apply Theorem B, we need to know that every morphism $d_1\to d_2$ in $\cat D^H$ induces an equivalence $\abs{d_2\!\downarrow\!F^H}\to \abs{d_1\!\downarrow\!F^H}$. This equivalence holds by assumption, since if $d_1$ and $d_2$ are objects of $\mathcal D^H$, then $H\leq G_{d_1}\cap G_{d_2}$.  Hence Theorem B allows us to conclude that for any object $d$ of $\mathcal D^H$, the pullback
\[ \begin{tikzcd}
\abs{d\!\downarrow\!F^H} \ar[r]\ar[d] & \abs{\mathcal C^H}\ar[d, "F^H"] \\
\abs{d\!\downarrow\!F^H}\ar[r]  & \abs{\mathcal D^H}
\end{tikzcd} \] 
is a homotopy pullback, i.e. the inclusion $\abs{d\!\downarrow\!F^H}\to {\rm hfib}_d(\abs{F^H})$ is a homotopy equivalence. Then \cref{lem:hfib of G-functor} implies $\abs{d\!\downarrow\!F}^H\to {\rm hfib}_d(\abs{F})^H$ is an equivalence. This argument applies to any $H\leq G_d$, and therefore $\abs{d\!\downarrow\!F}\to {\rm hfib}_d(\abs{F})$ is a $G_d$-equivalence.
\end{proof}

\begin{remark}
There is a dual version of Theorem B for $F\!\downarrow\!d$, where we instead assume each morphism $d\to d'$ in $\cat D$ induces a $G_d\cap G_{d'}$-equivalence $\abs{F\!\downarrow\!d}\to \abs{F\!\downarrow\!d'}.$
\end{remark}

\begin{remark}
As is true non-equivariantly, the equivariant version Theorem B could be used to give an alternative proof of the equivariant version of Theorem A.
\end{remark}

\bibliographystyle{alpha}
\bibliography{abbr-references}

\end{document}